\documentclass[11pt,a4paper]{article}

\usepackage{amssymb}
\usepackage{amsmath}
\usepackage{amsthm}
\usepackage[normalem]{ulem}
\usepackage{xcolor,cancel}

\usepackage[T1]{fontenc}
\usepackage[latin1]{inputenc}
\usepackage{color}

\usepackage{amsfonts}
\usepackage{amssymb}

\usepackage{empheq}

\usepackage[T1]{fontenc}   	% use T1 fonts (in particular, 8 bit encoding)
\usepackage{lmodern}       	% Latin Modern font (refinement of Computer Modern)
\usepackage{microtype}     	% better typesetting
\usepackage{hyperref}      	% hyperlinks
\usepackage{verbatim}	   	% To make comments
\usepackage{mathrsfs}		% To produce nice calligraphic letters

\usepackage{geometry}
\usepackage{graphicx}
\usepackage{subfigure}
\usepackage{tikz}
\usepackage{xcolor}
\usepackage{placeins}
\usepackage{wrapfig}
\usepackage{caption}
\usepackage{setspace}
\usepackage[capposition=top]{floatrow}

\usepackage{gensymb}	% temperature degree sign

\swapnumbers

\theoremstyle{plain}
\newtheorem{theorem}{Theorem}[section]
\newtheorem{lemma}[theorem]{Lemma}

\newtheorem{proposition}[theorem]{Proposition}

\theoremstyle{definition}
\newtheorem{definition}[theorem]{Definition}
\newtheorem{problem}[theorem]{Problem}

\theoremstyle{remark}
\newtheorem{remark}[theorem]{Remark}

\numberwithin{equation}{section}

\DeclareMathOperator*{\trace}{trace}
\DeclareMathOperator{\diam}{diam}
\DeclareMathOperator{\conv}{Conv}
\DeclareMathOperator{\area}{Area}

\DeclareMathOperator{\Div}{div}
\DeclareMathOperator{\vol}{Vol}

\title{Existence and Regularity of Spheres Minimising the Canham-Helfrich Energy}
\author{Andrea Mondino\thanks{University of Oxford. Mathematical Institute.   (UK). Email: Andrea.Mondino@maths.ox.ac.uk
} \and Christian Scharrer\thanks{University of Warwick. Mathematics Institute. Coventry  (UK). Email: C.Scharrer@warwick.ac.uk} }

\begin{document} 
\maketitle
\begin{abstract}
%Most cell membranes of living organisms are made of lipid bilayer, which is a thin polar membrane consisting of two opposite oriented layers of lipid molecules. 
%In the early 70s, Canham and Helfrich introduced their energy which is a linear combination of the   integrated squared mean curvature, called Willmore energy, the first moment of the mean curvature, and the area. 
We prove existence and regularity of minimisers for the Canham-Helfrich energy in the class of weak (possibly branched and bubbled) immersions of the $2$-sphere. This solves (the spherical case) of the minimisation problem proposed by Helfrich in 1973, modelling  lipid bilayer membranes.   On the way to prove the main results we establish the lower semicontinuity of the  Canham-Helfrich energy under weak convergence of  (possibly branched and bubbled) weak immersions.
\end{abstract}

\section{Introduction}
\label{sec:intro}

The basic structural and functional unit of all known living organisms is the cell. The interior material of a cell, the cytoplasm, is enclosed by biological membranes. Most of the cell membranes of living organisms are made of a lipid bilayer, which is a thin polar membrane consisting of two opposite oriented layers of lipid molecules. %These molecules have a hydrophilic head and a hydrophobic tail. Exposed to water, they self-assemble into a two-layered sheet with the hydrophobic tails pointing toward the centre of the sheet. 

In 1970, in order to explain the biconcave shape of red blood cells, Canham \cite{canham1970minimum} proposed a bending energy density dependent on the squared mean curvature. 

Three years later Helfrich proposed the following curvature elastic energy per unit area of a closed lipid bilayer \cite[Equation (12)]{helfrich1973elastic}
\begin{equation}\label{intro:Helfrich-density}
\frac{1}{2}k_c(H - c_0)^2 + \bar k_c K,
\end{equation} 
where $H$ is the mean curvature, $K$ is the Gauss curvature, $c_0$ is the so-called \emph{spontaneous curvature}, and $k_c, \bar k_c$ are the curvature elastic moduli. The values of the parameters can be measured experimentally (e.g. see \cite{evans1972improved}, \cite{deuling1976red} for   $c_{0}$, and \cite{mutz1990bending}  for $k_{c}$). 
%Based on experimental data of Evans and Fung \cite{evans1972improved} on red blood cells, Deuling and Helfrich \cite{deuling1976red} found approximately $c_0 = -0.74 \mu m^{-1}$. Mutz and Helfrich \cite{mutz1990bending} measured $k_c$ for dimyristoyl-phosphatidyl-ethanolamine membranes to be $k_c = 1.7 \times 10^{-12}$ at $T = 60 \degree \mathrm C$. 
The constant $\bar{k}_c$ is not important for the purpose of this paper as by the Gauss-Bonnet Theorem, the integrated Gauss curvature is a topological constant.

%One year later, Evans \cite[Equation (14)]{evans1974bending} computed the variation in free enerFpure bending of a bilayer. Similarly, the main ingredients in such a formula are the mean curvature and the squared mean curvature. 

Lipid bilayers are very thin compared to their lateral dimensions, thus are usually modelled as surfaces. Suppose the surface and hence the membrane is represented by a smooth isometric embedding $\vec \Phi: \mathbb S^2 \to \mathbb R^3$ of the $2$-sphere $\mathbb S^2$. We will be concerned with the following integrated version of \eqref{intro:Helfrich-density}
\begin{equation} \label{intro:Helfrich-energy}
	\mathcal H^{c_0}(\vec \Phi) := \int_{ \mathbb S^2 } (H - c_0)^2 \, d\mu = \int_{ \mathbb S^2 } \left(H^2 -2c_0H +1\right) \, d\mu
\end{equation}
where again, $H$ is the mean curvature, $c_0$ is a constant, and $\mu$ is the Radon measure corresponding to the pull back of the Euclidean metric along $\vec \Phi$. The integral in~\eqref{intro:Helfrich-energy} is known as \emph{Canham-Helfrich energy}. It is also referred to as \emph{Canham-Evans-Helfrich} or just \emph{Helfrich} energy. Its most important reduction is the \emph{Willmore energy}, where $c_0 = 0$. Due to its simplicity and fundamental nature, the Willmore energy appears in many areas of science and technology, and has been studied a lot in the past. Its first appearances were found in the works of Poisson \cite{poisson1814memoire} in 1814 and Germain \cite{germain1821recherches} in 1821. It was finally brought onto physical grounds by Kirchhoff \cite{MR1578677} in 1850 as the free energy of an elastic membrane. In the early 20th century, Blaschke considered the Willmore energy in the context of differential geometry and proved its conformal invariance, see for instance \cite{MR0076373}. 

The difference between the Willmore energy and the Canham-Helfrich energy comes from the constant~$c_0$, known as spontaneous curvature. According to Seifert \cite{seifert1997configurations}, it is mainly caused by asymmetry between the two layers of the membrane. Geometrically, the asymmetric area difference between the two layers is given by the total mean curvature, i.e. the integrated mean curvature. 
This is due to the fact that the infinitesimal variation of the area, i.e.  the area difference between two nearby surfaces,  is the total mean curvature.  D\"obereiner et al. \cite{dobereiner1999spontaneous} observed that spontaneous curvature may also arise from differences in the chemical properties of the aqueous solution on the two sides of the lipid bilayer. Many approaches about how to derive the Canham-Helfrich energy density as the energy density of a lipid bilayer have appeared in the literature. We refer to Seifert \cite{seifert1997configurations} for more details.

Our goal is to minimise the Canham-Helfrich energy as well as to study the regularity of minimisers (and more generally of critical points). In the language of the calculus of variations we are concerned with the following Problem \ref{prob:minimisation} stated in Bernard, Wheeler and Wheeler \cite[Introduction, Problem (P1)]{MR3757086}. Given a smooth embedding $\vec \Phi: \mathbb S^2 \to \mathbb R^3$ denote by $\area \vec \Phi=\int_{\mathbb S^2} d\mu$ the area of  the surface $\vec \Phi(\mathbb S^2)$ and by  $\vol \vec \Phi$ the enclosed volume. 

\begin{remark}
	A candidate embedding $\vec \Phi_0$ which achieves the global minimum is called a \emph{minimiser}. In general  it is not unique and, more dramatically, it may not exist: later in the introduction we show that for a suitable choice of parameters the minimum is achieved by a singular immersion and it cannot be achieved by a smooth one. The constraints and the functional $\mathcal H^{c_0}$ are invariant under reparametrisation as well as rigid motions in $\mathbb R^3$. Of course, in order to have a non-empty class of competitors, the constraints have to satisfy the Euclidean isoperimetric  inequality $A_0^3 \geqslant 36 \pi V_0^2$.
\end{remark}

\begin{problem} \label{prob:minimisation}
	Let $c_0, A_0$, and $V_0$ be given constants. Minimise $\mathcal H^{c_0}(\vec \Phi)$ in the class of smooth embeddings $\vec \Phi: \mathbb S^2 \to \mathbb R^3$ subject to the constraints
	\begin{equation} \label{intro:constraints}
		\area \vec \Phi = A_0 \quad \text{and} \quad \vol \vec \Phi = V_0.
	\end{equation}
	That is, find an embedding $\vec \Phi_0 : \mathbb S^2 \to \mathbb R^3$ such that $\area \vec \Phi_0 = A_0$, $\vol \vec \Phi_0 = V_0$, and
	\begin{equation*}
		\mathcal H^{c_0} (\vec \Phi_0) \leqslant \mathcal H^{c_0}(\vec \Phi)
	\end{equation*}
	for any other smooth embedding $\vec \Phi: \mathbb S^2 \to \mathbb R^3$ satisfying the constraints \eqref{intro:constraints}.
\end{problem}

Problem \ref{prob:minimisation} is the classical formulation suggested in \cite{helfrich1973elastic} and \cite{deuling1976red}. According to Bernard, Wheeler and Wheeler \cite{MR3757086}, many issues for the Canham-Helfrich energy, including Problem~\ref{prob:minimisation}, remain open and form important questions that future research should address. A similar problem in the $2$-dimensional case (i.e. closed curves in the Euclidean plane) was formulated and solved by Bellettini, Dal Maso, and Paolini \cite{MR1233638} by a relaxation procedure.

While there was essentially no work on the variational theory of the Willmore energy after Blaschke's seminal work, in 1965 Willmore \cite {MR0202066} reintroduced this Lagrangian which is now named after him. He showed that the round sphere is a minimiser of Problem \ref{prob:minimisation} in the special case $c_0 = 0$ without constraints \eqref{intro:constraints}, see \cite{MR686105}. Simon \cite{MR857667} proved existence of higher genus minimisers for the Willmore energy (see also Kusner \cite{MR1417949} and Bauer-Kuwert \cite{MR1941840}), using the so-called \emph{ambient approach}, i.e. convergence of surfaces is considered in the measure-theoretic sense. Rivi\`ere \cite{MR2430975, MR3276154} proved the analogous result with the so called \emph{parametric approach}, i.e. based on PDE theory and functional analysis as opposed to geometric measure theory. In the present paper we shall adapt the parametric approach. The case $c_0 = 0$ with constraints \eqref{intro:constraints} was solved by Schygulla \cite{MR2928137} using the ambient approach, and generalised to higher genus surfaces by Keller, the first author, and Rivi\`ere \cite{MR3176354} using the parametric approach. 
 
From the mathematical point of view, the spontaneous curvature $c_0$ causes a couple of differences between the Willmore energy and the Canham-Helfrich energy. Most obviously, the Canham-Helfrich energy cannot be bounded below by a strictly positive constant, whereas the Willmore energy is bounded below by $4\pi$, see \eqref{pre:lower_bound_Willmore}. Secondly, while the Willmore functional is invariant under conformal transformations, the Canham-Helfrich energy is not conformally invariant. We will be concerned with yet another property that fails for the Canham-Helfrich energy due to non-negative spontaneous curvature. Namely lower semi-continuity with respect to varifold convergence: while it is well known that the Willmore functional is lower semi-continuous under varifold convergence, the Canham-Helfrich energy in general is not. Indeed, Gro{\ss}e-Brauckmann \cite[\emph{Remark} (ii) on page~550]{MR1248112} constructed a sequence of non-compact infinite genus surfaces $\Sigma_1, \Sigma_2, \ldots$ with constant mean curvature equal to $1$ which converges in the varifold sense to a double plane $\Sigma_\infty$. Hence, the mean curvature $H_\infty$ of the limit $\Sigma_\infty$ is zero and
\begin{equation*}
0 = \int_{\Sigma_k} (H_k - 1)^2\eta \,d\mathscr H^2 < 2 \int_{\Sigma_\infty} (H_\infty-1)^2\eta \,d\mathscr H^2 = 2 \int_{\Sigma_\infty} \eta\,d\mathscr H^2   
\end{equation*} 
for any continuous non-negative, non-zero function $\eta$ on $\mathbb R^3$ of compact support, where $\mathscr H^2$ is the $2$-dimensional Hausdorff measure. Hence, the general Canham-Helfrich energy is not lower semi-continuous under varifold convergence. However, in order to solve Problem \ref{prob:minimisation} by the so-called \emph{direct method of calculus of variations}, lower semi continuity is required. According to R\"oger \cite{MR2524083}, it was an open question under which conditions/in which natural weak topology on the space of immersions one obtains lower semi continuity of the Canham-Helfrich energy. This is presumably the reason why Problem \ref{prob:minimisation} was only partially solved for non-zero spontaneous curvature $c_0$. The axisymmetric case was solved in 2013 by Choksi and Veneroni \cite {MR3116014}, who proved the existence of a minimiser of the Canham-Helfrich energy  among a suitable class of axisymmetric (possibly singular) surfaces under  fixed surface area and  enclosed volume constraints.  Five years later, Dalphin \cite{MR3827803} showed existence of minimisers in a class of $C^{1,1}$ surfaces whose principal curvatures are bounded by a given constant $1/\varepsilon$. Though, in his setting, it is still unclear how to get compactness and lower semi-continuity as $\varepsilon$ tends to zero.

As already alluded to, we tackle Problem \ref{prob:minimisation} by the direct method of calculus of variations. The issue is of course to find a suitable  class $\mathcal F_{A_0, V_0}$ of admissible maps (endowed with a suitable topology)  having area $A_0$ and enclosed volume $V_0$ such that the Canham-Helfrich energy is lower semi-continuous and has (pre-)compact sub-levels.
%
% Therefore we start with a sequence $\vec \Phi_1, \vec \Phi_2, \ldots$ of maps $\mathbb S^2 \to \mathbb R^3$ that belong to a certain class $\mathcal F_{A_0, V_0}$ of admissible maps having area $A_0$ and enclosed volume $V_0$ such that
%\begin{equation*}
% 	\mathcal H^{c_0}(\vec \Phi_k) \to \inf_{\vec \Phi \in \mathcal F_{A_0, V_0}} \mathcal H^{c_0}(\vec \Phi) \qquad \text{as } k \to \infty.
%\end{equation*}  
%Then, we extract a subsequence $\vec \Phi_{k_l}$ and a limit $\vec T$ such that
%\begin{equation} \label{intro:compactness}
%	\vec \Phi_{k_l} \to \vec T \qquad \text{as } l \to \infty.
%\end{equation}
%This step is referred to as \emph{compactness}. The key point is of course to find a suitable class of admissible maps~$\mathcal F$ as well as a notion of convergence in \eqref{intro:compactness}. Then, we show that
%\begin{equation*} 
%	\mathcal H^{c_0}(\vec T) \leqslant \liminf_{l \to \infty} \mathcal H^{c_0}(\vec \Phi_{k_l}).
%\end{equation*}
%This step is referred to as \emph{lower semi-continuity}. It guarantees that the limit $\vec T$ is indeed a minimiser. In order to get compactness, we will have to choose $\mathcal F$ in a reasonably large class, allowing us to employ strong theorems from functional analysis such as the Banach-Alaoglu theorem. 
A natural choice is the class of weak (Sobolev) immersions in $W^{2,2}(\mathbb S^2, \mathbb R^3)$, already employed in the context of the Willmore energy for instance by Rivi\`ere \cite {MR3276154} or Kuwert and Li \cite{MR2928715}. We will use the space of  bubble trees of weak possibly branched immersions. It will shortly become clear why we have to allow branched points and multiple bubbles.

In the following, we denote by $ \boldsymbol \cdot$ (resp. $\times$) the Euclidean scalar (resp. vector) product on~$\mathbb R^{3}$.
\begin{definition} \label{def:weak_immersion}
	A map $\vec \Phi: \mathbb S^2 \to \mathbb R^3$ is called \emph{weak (possibly branched) immersion with finite total curvature} if  $\vec \Phi \in W^{1, \infty}(\mathbb S^2, \mathbb R^3)$ and  the following holds:
	\begin{enumerate}
	\item   There exists $C>1$ such that, for a.e. $p\in \mathbb S^{2}$, 
	\begin{equation} \label{intro:non-deg} 
	C^{-1} |d \vec \Phi|^{2}(p)\leq |d \vec \Phi \times d \vec \Phi|(p) \leq C |d \vec \Phi|^{2}(p),
	\end{equation} 
	where the norms are taken with respect to the standard metric on $\mathbb S^{2}$ and with respect to the Euclidean metric of $\mathbb R^{3}$, and where  $d \vec \Phi \times d \vec \Phi$ is the tensor given in local coordinates on $\mathbb S^{2}$ by 
	$$
	d \vec \Phi \times d \vec \Phi:= 2 \partial_{x^{1}} \vec \Phi \times  \partial_{x^{2}} \vec \Phi  \; dx^{1}\wedge dx^{2} \in   \wedge^{2} T^{*} \mathbb S^{2} \otimes \mathbb R^{3};
	$$ 
	\item  There exist a positive integer $N$ and finitely many points $b_1, \ldots, b_N \in \mathbb S^2$ such that $\log |d \vec \Phi|\in L^{\infty}_{loc} (\mathbb S^2 \setminus \{b_1, \cdots, b_N\})$;
        \item   The Gauss map $\vec n$, defined by 
	\begin{equation*}
	\vec n := \frac{\partial_{x^1} \vec \Phi \times \partial_{x^2} \vec \Phi}{|\partial_{x^1} \vec \Phi \times \partial_{x^2} \vec \Phi|}
	\end{equation*}
	in any local chart $x$ of $\mathbb S^2$, satisfies 
	\begin{equation*}
	\vec n \in W^{1,2}(\mathbb S^2, \mathbb R^3).
	\end{equation*}
	\end{enumerate}
	The space of weak (possibly branched) immersions with finite total curvature is denoted by $\mathcal F$.		
\end{definition}

Since by assumption $\vec \Phi$ is a Lipschitz map, it induces an $L^\infty$-metric $g$ given  by
\begin{equation*}
	g(X,Y) = d\vec \Phi(X) \boldsymbol \cdot d\vec \Phi(Y)
\end{equation*}
for elements $X,Y$ of the tangent bundle $T\mathbb S^2$. In the usual way (see for instance \cite[1.2]{MR1688256}), the $L^\infty$-metric $g$ induces a Radon measure $\mu$ on $\mathbb S^2$ which is mutually absolutely continuous to the $2$-dimensional Hausdorff measure on $\mathbb S^2$. 

Using M\"uller-Sv\v{e}r\'ak theory of weak isothermic charts \cite{MR1366547}  and H\'elein's moving frame technique \cite{MR1913803} one can prove the following proposition (see for instance \cite{MR3524220})
\begin{proposition}
\label{pr-I.1}
Let $\vec{\Phi}\in \mathcal F$ be a  weak (possibly branched) immersion of $\mathbb S^2$ into $\mathbb R^{3}$. Then there exists a bilipschitz homeomorphism $\Psi$ of $\mathbb S^2$
such that $\vec{\Phi}\circ\Psi$ is weakly conformal: it satisfies almost everywhere on $\mathbb S^2$
$$
\left\{
\begin{array}{l}
 |\partial_{x^{1}}(\vec{\Phi}\circ\Psi)|^{2}=|\partial_{x^{2}}(\vec{\Phi}\circ\Psi)|^2\quad\\
 \partial_{x^{1}}(\vec{\Phi}\circ\Psi)  \boldsymbol \cdot \partial_{x^{2}}(\vec{\Phi}\circ\Psi)=0,
\end{array}
\right.
$$
where $x$ is a local arbitrary conformal chart on $\mathbb S^2$ for the standard metric. Moreover $\vec{\Phi}\circ\Psi$ is in $W^{1,\infty}(\mathbb S^2, \mathbb R^{3}) \cap W^{2,2}_{loc} (\mathbb S^2 \setminus \{b_{1},\ldots, b_{N} \},  \mathbb R^{3} )$.
\end{proposition}

\medskip

\begin{remark}
\label{rm-I.1}
In view of  Proposition~\ref{pr-I.1}, a careful reader could wonder why we do not work  with \emph{conformal} $W^{2,2}$ weak, possibly branched, immersions only and why we
do not impose for the membership in ${\mathcal F}$, $\vec{\Phi}$ to be conformal from the beginning. The reason why it is technically convenient not to impose conformality from the beginning is to allow general perturbations in the variational problem, which do not have to respect infinitesimally the conformal condition.  
\end{remark} 

The reason why we chose the class $\mathcal F$ as above is the following theorem of the first author and Rivi\`ere \cite[Theorem 1.5]{MR3276119} (see also \cite{ChenLi}). 

\begin{theorem} \label{thm:compactness}
	Suppose $\vec\Phi_1, \vec \Phi_2,\ldots$ is a sequence in $\mathcal F$ of conformal  weak (possibly branched) immersions  such that
	\begin{equation*}
		\limsup_{k\to\infty}\int_{\mathbb S^2}1+|d \vec n_{k}|^2\,d\mu_{k} < \infty, \qquad \liminf_{k\to\infty} \diam \vec\Phi_k[\mathbb S^2]>0
	\end{equation*}
	where $\vec n_{k}$ are the Gauss maps, $\mu_{k}$ are the corresponding Radon measures, and $\diam \vec\Phi_k[\mathbb S^2] := \sup_{a, b \in \mathbb S^2}|\vec \Phi(a) - \vec \Phi (b)|$.
	
	Then, after passing to a subsequence, there exist a family $\Psi_k$ of bilipschitz homeomorphisms of $\mathbb S^2$, a positive integer $N$, sequences $f_k^1,\ldots,f_k^N$ of positive conformal diffeomorphisms of $\mathbb S^2$, $\vec \xi_\infty^1,\ldots,\vec \xi_\infty^N \in \mathcal F$, non-negative integers $N_1,\ldots,N_N$, and finitely many points on the sphere 
	\begin{equation*}
	\{b^{i,j}:j=1,\ldots,N_i,\,i=1,\ldots, N\}\subset \mathbb S^2
	\end{equation*}	
	such that
	\begin{equation} \label{intro:continuous_convergence}
		\vec \Phi_k \circ \Psi_k \to \vec f_\infty \qquad \text{as } k \to \infty \text{ strongly in } C^0(\mathbb S^2, \mathbb R^3)
	\end{equation}
	for some $\vec f_\infty \in W^{1, \infty}(\mathbb S^2, \mathbb R^3)$ and
	\begin{equation*} 
		\vec \Phi_k \circ f^i_k \rightharpoonup \vec \xi_ \infty^i\qquad \text{as }k\to \infty \text{ weakly in } W^{2,2}_{\mathrm{loc}}(\mathbb S^2\setminus\{b^{i,1},\ldots,b^{i,N_i}\},\mathbb R^3)
	\end{equation*}
	for $i = 1,\ldots,N$. Moreover,
	\begin{equation*} 
		\sum_{i = 1}^N \int_ {\mathbb S^2}1\, d\mu_ {\vec \xi_\infty^i} = \lim_{k \to \infty} \int_{\mathbb S^2} 1 \, d\mu_{k}. 
	\end{equation*} 
\end{theorem}
The theorem already gives (pre-)compactness, a notion of convergence, and lower semi-continuity (actually, continuity) of the third summand in \eqref{intro:Helfrich-energy} of the Canham-Helfrich energy, i.e. of the area functional. Indeed, $\vec T := (\vec f_\infty, \vec \xi_\infty^1, \ldots, \vec \xi_\infty^N)$ forms a \emph{bubble tree}, see Definition \ref{def:bubble_tree}. In particular, the limit $\vec T$ is not in the class $\mathcal F$ anymore. 
At an informal level, a non expert reader can think of a bubble tree $\vec T := (\vec f, \vec \xi^1, \ldots, \vec \xi^N)$  as a ``pearl necklace'' where each ``pearl'' corresponds to the image of a  possibly branched weak immersion  $\vec \xi^i(\mathbb S^{2})$   and $\vec f$ is a Lipschitz map from $\mathbb S^{2}$ to $\mathbb R^{3}$ ``parametrising'' the whole pearl necklace, in particular  $\vec f(\mathbb S^{2})=\bigcup_{i=1}^{N} \vec \xi^i(\mathbb S^{2})$.
\\To get a better understanding of why we obtain a bubble tree in the limit, we will look at an example of Problem \ref{prob:minimisation}. Let
\begin{equation*}
	c_0 = 1, \qquad A_0 = 2\area \mathbb S^2, \qquad V_0 = 2 \vol \mathbb S^2. 
\end{equation*}
Then, the infimum in  Problem \ref{prob:minimisation} is achieved by the bubble tree $\vec T = (\vec f,  \vec{{\rm Id}}_{\mathbb S^2}, \vec{{\rm Id}}_{\mathbb S^2})$ of twice the unit sphere. Indeed, $\mathcal H^{c_0}(\vec T) = 0$ and $\mathcal H^{c_0}(\vec \Phi) \geq 0$ for any other smooth immersion $\vec \Phi$ of $\mathbb S^2$ into $\mathbb R^3$, so $\vec T$ achieves the infimum. A minimising sequence $\vec{\Phi}_{k}(\mathbb S^{2})$ of smoothly embedded spheres converging to such a bubble tree can be achieved by glueig $(1+ 1/k) {\mathbb S}^{2}$ to $(1- 1/k) {\mathbb S}^{2}$ via a small catenoidal neck of size $2/k$.
\\ Notice also that if $\vec \Phi$ satisfies $\mathcal H^{c_0}(\vec \Phi) = 0$, then the image $\vec \Phi[\mathbb S^2]$ is the unit sphere by a classical theorem of Hopf \cite {MR707850}. 

Getting a bubble tree in the limit is in accordance with the earlier result on existence of minimisers by Choksi and Veneroni \cite {MR3116014} in the axisymmetric case: indeed the minimiser in \cite[Theorem 1] {MR3116014}  is made by a finite union of axisymmetric surfaces.
\\Moreover, the bubbling phenomenon  is also known as \emph{budding transition} in biology and has been recorded with video microscopy, see Seifert \cite{seifert1997configurations} or Seifert, Berndl, and Lipowsky \cite{seifert1991shape}.

In Chapter \ref{sec:existence} we sharpen Theorem \ref{thm:compactness} in a way that we get lower semi-continuity for the Canham-Helfrich functional. This can be seen as a possible answer to the aforementioned open question raised by R\"oger  \cite{MR2524083}.  In Chapter \ref{sec:regularity} we compute the Euler-Lagrange equation for the Canham-Helfrich energy in divergence form. Moreover, we prove that all the weak branched conformal immersions of a minimising  bubble tree (actually more generally for a critical bubble tree) are smooth away from their branch points. Our proof is based on the regularity theory for Willmore surfaces developed by Rivi\`ere \cite{MR2430975}. It relies on conservation laws discovered by Rivi\`ere \cite{MR2430975} in the context of the Willmore energy and adjusted by Bernard \cite{MR3518329} for the Canham-Helfrich energy. We get the following final result.

\begin{theorem} \label{thm:minimisation}
	Suppose $c_0 \in\mathbb R$, $A_0, V_0 > 0$, and $A_0^3 \geqslant 36 \pi V_0^2$. 
	
	Then, there exist a positive integer $N$ and weak branched conformal immersions of finite total curvature $\vec \Phi_1, \ldots, \vec \Phi_N \in \mathcal F$ such that   $\cup_{i=1}^{N} \vec \Phi_{i}(\mathbb S^{2})$ is connected, 
	\begin{equation} \label{min}
	 \inf_{\substack{\vec \Phi \in \mathcal F \\ \area \vec \Phi = A_0 \\ \vol \vec \Phi = V_0}} \mathcal H^{c_0}(\vec \Phi) = \sum_{i = 1}^N \mathcal H^{c_0}(\vec \Phi_i) 
	\end{equation}
	and 
	\begin{equation*} 
	 \sum_{i = 1}^N \area \vec \Phi_i = A_0, \qquad \sum_{i = 1}^N \vol \vec \Phi_i = V_0.
	\end{equation*}
	
	Moreover, for each $i \in \{1, \ldots, N\}$ there exist a non-negative integer $N^i$ and finitely many points $b^{i,1},\ldots,b^{i,N^i} \in \mathbb S^2$ such that $\vec \Phi_i$  is a $C^{\infty}$ immersion of $\mathbb S^2 \setminus \{b^{i,1}, \ldots, b^{i, N^i}\}$ into $\mathbb R^3$ and $b^{i,1},\ldots,b^{i,N^i}$ are branch points for $\vec \Phi_i$.

	Furthermore, there exists a constant $\varepsilon_{\ref{lem:embeddedness}}(A_0, V_0)>0$ such that if $|c_0| < \varepsilon_{\ref{lem:embeddedness}}(A_0, V_0)$, then $N = 1$ and $\vec \Phi := \vec \Phi_1$ is a smooth embedding of $\mathbb S^2$ into $\mathbb R^3$.
\end{theorem} 

\begin{proof}
	Let $\vec \Phi_1, \vec \Phi_2, \ldots$ be a minimising sequence of \eqref{min}. There holds
	\begin{equation} \label{pm:bound_Willmore}
		\int_{ \mathbb S^2 } H_{\vec \Phi_k}^2 \, d\mu_{\vec \Phi_{k}} = \int_{ \mathbb S^2 } 2(H_{\vec \Phi_{k}} - c_0)^2 - (H_{\vec \Phi_k} - 2c_0)^2 + 2c_0^2 \, d\mu_{\vec \Phi_{k}} \leqslant 2\mathcal H^{c_0}(\vec \Phi_k) +2c_0^2A_0
	\end{equation}
	where $H_{\vec \Phi_k}$ is the mean curvature corresponding to $\vec \Phi_k$, see \eqref{pre:mean_curvature}. By the Gauss-Bonnet theorem (see \eqref{eq:GaussBonnetBranchedImm} for the precise statement in case of weak branched immersions and \eqref{eq:ControlL2II} for the estimate below),
	\begin{equation*}
		 \int_{ \mathbb S^2 } |d \vec n_{\vec \Phi_{k}}|^2 \, d\mu_{\vec \Phi_{k}} \leq 4 \int_{ \mathbb S^2 }H_{\vec \Phi_k}^2 \, d\mu_{\vec \Phi_{k}} 
	\end{equation*} 
	and thus 
	\begin{equation*}
		\int_{ \mathbb S^2 } 1 + |d \vec n_{\vec \Phi_{k}}|^2 \, d\mu_{\vec \Phi_{k}} \leqslant 8\mathcal H^{c_0}(\vec \Phi_k) + (1+8c_0^2)A_0
	\end{equation*}	
	which means the first inequality of \eqref{wc:uniformly_bounded_energy} is satisfied. Moreover, \eqref{pre:diameter_bound} implies the second inequality of \eqref{wc:uniformly_bounded_energy}. Hence, we can apply Theorem \ref{thm:compactness_and_lsc}, Theorem \ref{thm:smoothness}, Lemma \ref{lem:embeddedness},  \eqref{pre:lower_bound_Willmore} and \eqref{pre:Li-Yau-inequality} to conclude the proof.
\end{proof}

\begin{remark}
The arguments in the proof of  Theorem \ref{thm:minimisation} yield also that the minimum of $\mathcal H^{c_0}$ is achieved in the class of bubble trees of possibly branched weak immersions,  by a bubble tree of possibly branched immersions which are smooth out of the branch points.
\end{remark}

A more general form of the Canham-Helfrich energy is given by
\begin{equation*}
\mathcal{H}^{c_0}_{\alpha, \rho}(\vec \Phi) := \int_{\mathbb S^2} (H_{\vec \Phi} - c_0)^2 \, d\mu_{\vec \Phi} + \alpha \area \vec \Phi + \rho \vol \vec \Phi,
\end{equation*}
for $\vec \Phi \in \mathcal F$ where the parameter $\alpha \geqslant 0$ is referred to as \emph{tensile stress}, and $\rho \geqslant 0$ as \emph{osmotic pressure}. We get the following solution of Problem (P2) from the introduction in \cite{MR3757086}.

\begin{theorem}
	Suppose $c_0 \in \mathbb R$, $\alpha > 0$, and $\rho \geqslant 0$. Then, there holds
	\begin{equation*}
		\inf_{\vec{\Phi} \in \mathcal F} \mathcal H^{c_0}_{\alpha, \rho}(\vec \Phi) \leqslant 4\pi.
	\end{equation*}
	Moreover, if the inequality is strict, then there exist $\vec \Phi_0 \in \mathcal F$, a positive integer $N_0$, and points $b_1, \ldots, b_{N_0} \in \mathbb S^2$ such that
	\begin{equation*}
		\inf_{\vec{\Phi} \in \mathcal F} \mathcal H^{c_0}_{\alpha, \rho}(\vec \Phi) = \mathcal H^{c_0}_{\alpha, \rho}(\vec \Phi_0), 
	\end{equation*}
	$\vec \Phi_0$  is a $C^{\infty}$ immersion of $\mathbb S^2 \setminus \{b^{i,1}, \ldots, b^{i, N^i}\}$ into $\mathbb R^3$ and $b^{i,1},\ldots,b^{i,N^i}$ are branch points.
	\\Furthermore,  if $|c_0| < \sqrt{\alpha}$, then $\vec \Phi_0$ is a smooth embedding. 
\end{theorem}

\begin{proof}
	Taking $\vec \Phi_k = \frac{1}{k} \mathbb S^2$ for each integer $k$ leads to  
	\begin{equation*}
		\inf_{\vec{\Phi} \in \mathcal F} \mathcal H^{c_0}_{\alpha, \rho}(\vec \Phi)  \leqslant \liminf_{k\to \infty}  \mathcal H^{c_0}_{\alpha, \rho}(\vec \Phi_{k}) =  \int_{ \mathbb S^2 } H_{\mathbb S^2}^2 \, d\mu_{\mathbb S^2} = 4\pi,
	\end{equation*} 
	which proves the first statement. 
	Now assume $\vec \Phi_1, \vec \Phi_2, \ldots$ is a sequence in $\mathcal F$ such that
	\begin{equation*}
		\lim_{k \to \infty} \mathcal H^{c_0}_{\alpha, \rho}(\vec \Phi_k) = \inf_{\vec{\Phi} \in \mathcal F} \mathcal H^{c_0}_{\alpha, \rho}(\vec \Phi) < 4\pi.
	\end{equation*}
	As  $\alpha > 0$, we have $\sup_{k} \area \vec \Phi_{k}<\infty$ and thus, using \eqref{pm:bound_Willmore}, also  $\sup_{k} \int_{ \mathbb S^2 } H_{\vec \Phi_k}^2 \, d\mu_{\vec \Phi_{k}}<\infty$. 
	\\A simple contradiction argument using \eqref{pre:diameter_bound} now leads to 
	\begin{equation*}
		\liminf_{k \to \infty} \diam \vec \Phi_k[\mathbb S^2] > 0,
	\end{equation*}
	as otherwise we had
	\begin{equation*}
		\lim_{k \to \infty} \area \vec \Phi_k = 0, \qquad \lim_{k \to \infty} \mathcal H^{c_0}_{\alpha, \rho}(\vec \Phi_k) = \lim_{k \to \infty} \int_{ \mathbb S^2 } H_{\vec \Phi_k}^2 \, d\mu_{\vec \Phi_{k}} \geqslant 4\pi.
	\end{equation*}
	Therefore, analogously to the proof of Theorem \ref{thm:minimisation}, we can apply Theorem \ref{thm:compactness_and_lsc} to obtain an integer $N$ and $\vec \Phi_1, \ldots, \vec \Phi_N \in \mathcal F$ such that 
	\begin{equation*}
		\inf_{\vec{\Phi} \in \mathcal F} \mathcal H^{c_0}_{\alpha, \rho}(\vec \Phi) = \sum_{i = 1}^N\mathcal H^{c_0}_{\alpha, \rho}(\vec \Phi_i).
	\end{equation*}
	Obviously,
	\begin{equation*}
		\mathcal H^{c_0}_{\alpha, \rho}(\vec \Phi_1) \leqslant \sum_{i = 1}^N\mathcal H^{c_0}_{\alpha, \rho}(\vec \Phi_i)
	\end{equation*}
	and since there are no constraints, we simply get $N=1$. Letting $\vec \Phi_0 := \vec \Phi_1$, we infer from \eqref{pm:bound_Willmore} that in case $|c_0| \leqslant \sqrt{\alpha}$
	\begin{equation*}
		\int_{ \mathbb S^2 } H_{\vec{\Phi}_0}^2 \, d\mu_{\vec \Phi_0} \leqslant 2\mathcal H^{c_0}(\vec \Phi_k) +2c_0^2 \area \vec \Phi_0[\mathbb S^2] \leqslant 2\mathcal H^{c_0}_{\alpha, \rho}(\vec \Phi_0) < 8\pi. 
	\end{equation*}
	The conclusion follows from Theorem \ref{thm:smoothness}, and \eqref{pre:Li-Yau-inequality}.
\end{proof}

\textbf{Acknowledgements.}
A.M. is supported by the EPSRC First Grant EP/R004730/1  ``Optimal transport and Geometric Analysis'' and by the ERC Starting Grant  802689 ``CURVATURE''.
\\ C.S. is supported by the EPSRC as part of the MASDOC DTC at the University of Warwick, grant No. EP/HO23364/1.

\section{Preliminaries}

\subsection{Notation}
We adopt the conventions of  \cite{MR3524220}. To avoid indices and to get clearly arranged equations, we will employ the following suggestive notation. For $\mathbb R^3$ valued maps $\vec e$ and $\vec f$ defined on the unit disk $D^2$, we write
\begin{gather*}
	\nabla \vec e := \begin{pmatrix} \partial_{x^1} \vec e \\ \partial_{x^2} \vec e \end{pmatrix}, \qquad \nabla^\bot \vec e := \begin{pmatrix} - \partial_{x^2} \vec e  \\  \; \; \partial_{x^1} \vec e \\ \end{pmatrix} \\
	\langle \vec e, \nabla \vec f \rangle := \begin{pmatrix} \vec e \boldsymbol \cdot \partial_{x^1} \vec f\\ \vec e \boldsymbol \cdot \partial_{x^2} \vec f\end{pmatrix}, \qquad \vec e \times \nabla \vec f := \begin{pmatrix} \vec e \boldsymbol \times \partial_{x^1} \vec f\\ \vec e \boldsymbol \times \partial_{x^2} \vec f\end{pmatrix}
\end{gather*}
as well as 
\begin{gather*}
	\nabla \vec e \times \nabla \vec f := \partial_{x^1} \vec e \times \partial_{x^1} \vec f + \partial_{x^2} \vec e \times \partial_{x^2} \vec f, \\
	\vec e \boldsymbol \cdot \nabla \vec f := \vec e \boldsymbol \cdot \partial_{x^1} \vec f + \vec e \boldsymbol \cdot \partial_{x^2} \vec f,
	\qquad \nabla \vec e \boldsymbol \cdot \nabla f := \partial_{x^1} \vec e \boldsymbol \cdot \partial_{x^1} \vec f + \partial_{x^2} \vec e \boldsymbol \cdot \partial_{x^2} \vec f, 
\end{gather*}
where $\boldsymbol \cdot$ denotes the Euclidean inner product and $\times$ denotes the usual vector product on $\mathbb R^3$. Similarly, for $\lambda: D^2 \to \mathbb R$ we write
\begin{equation*}
	\langle \nabla \lambda, \vec e \rangle := \begin{pmatrix} (\partial_{x^1} \lambda) \vec e\\ (\partial_{x^2} \lambda) \vec e \end{pmatrix}, \qquad \langle \nabla \lambda, \nabla \vec e \rangle := \partial_{x^1}\lambda \partial_{x^1} \vec e + \partial_{x^2}\lambda \partial_{x^2} \vec e.
\end{equation*}
Moreover, for a vector field 
\begin{equation*}
	\vec X = \begin{pmatrix} \vec X^1 \\ \vec X^2 \end{pmatrix}
\end{equation*}
with components $\vec X^1, \vec X^2: D^2 \to \mathbb R^3$, we define the divergence 
\begin{equation*}
	\Div \vec X := \partial_{x^1} \vec X^1 + \partial_{x^2} \vec X^2.
\end{equation*}
The $m$-dimensional Lebesgue measure is denoted by $\mathscr L^m$.

\subsection{Weak (possibly branched) conformal immersions} \label{subs:weak_immersions}
We adapt the notion of weak immersions which was independently formalized by Rivi\`ere \cite{MR3276154} and Kuwert and Li \cite{MR2928715}.

Let $(\Sigma, c_{0})$ be a smooth closed Riemann surface (in the rest of the paper we will take $(\Sigma, c_{0})$ to be the 2-sphere endowed with the standard round metric). Without loss of generality we can assume that  $(\Sigma, c_{0})$ is endowed with a metric $g_{c_{0}}$ of constant curvature and area $4\pi$ (see for instance \cite{MR2247485}).  For the definition of the Sobolev spaces $W^{k,p}(\Sigma, \mathbb R^3)$ on $\Sigma$ see for instance Hebey \cite{MR1688256}. A map $\vec \Phi: \Sigma \to \mathbb R^3$ is called a \emph{weak branched conformal immersion with finite total curvature} if and only if there exists a positive integer $N$, finitely many points $b_1, \ldots, b_N \in \Sigma$ such that
\begin{equation*}  
	\vec \Phi \in W^{1, \infty}(\Sigma, \mathbb R^3) \cap W_{\mathrm{loc}}^{2,2}(\Sigma \setminus \{b_1, \cdots, b_N\}, \mathbb R^3), 
\end{equation*} 
there holds
\begin{equation} \label{eq:conformal}
	\left\{ \begin{split}
	|\partial_{x^1} \vec \Phi| = |\partial_{x^2} \vec \Phi| \\
	\partial_{x^1} \vec \Phi \boldsymbol \cdot \partial_{x^2} \vec \Phi = 0
	\end{split} \right. 
\end{equation}
almost everywhere for any conformal chart $x$ of $\Sigma$, 
\begin{equation*}
	\log |d \vec \Phi| \in L^\infty_\mathrm{loc}(\Sigma \setminus \{b_1, \ldots, b_N\}),
\end{equation*}
and its Gauss map $\vec n$ defined by 
\begin{equation*}
\vec n := \frac{\partial_{x^1} \vec \Phi \times \partial_{x^2} \vec \Phi}{|\partial_{x^1} \vec \Phi \times \partial_{x^2} \vec \Phi|}
\end{equation*}
in any local positive chart $x$ of $\Sigma$ satisfies 
\begin{equation}\label{pre:finie_total_curvature}
	\vec n \in W^{1,2}(\Sigma, \mathbb R^3).
\end{equation}
The space of weak branched conformal immersions with finite total curvature is denoted by $\mathcal F_{\Sigma}$ or just $\mathcal{F}$ in case $\Sigma = \mathbb S^2$. We define the $L^\infty$-metric $g$ pointwise for almost every $p \in \Sigma$ by
\begin{equation*}
	g_p(X, Y) := d\vec \Phi_p(X) \boldsymbol \cdot d\vec \Phi_p(Y) 
\end{equation*}
for elements $X, Y$ of the tangent space $T_p\Sigma$. In the usual way, the $L^\infty$-metric $g$ induces a Radon measure $\mu_g$ on $\Sigma$. The conformality condition \eqref{eq:conformal} implies that $g=e^{2\lambda} g_{c_{0}}$ for some $\lambda \in L^{\infty}_{loc} (\Sigma \setminus \{b_{1}, \ldots, b_{N}\})$ called \emph{conformal factor}.   
Moreover, we define the second fundamental form $\vec{\mathbb I}$ pointwise for almost every $p \in \Sigma$ by
\begin{equation*}
	\vec{\mathbb I}_p: T_p\Sigma \times T_p \Sigma \to \mathbb R^3, \qquad \vec{\mathbb I}_p(X, Y) := -[d\vec n_p(X) \boldsymbol \cdot d\vec \Phi_p(Y)] \vec n.
\end{equation*}
The mean curvature vector $\vec H$ and the scalar mean curvature $H$ are given by
\begin{equation} \label{pre:mean_curvature}
	\vec H := \frac{1}{2}\trace \vec{\mathbb I}, \qquad H := \vec n \boldsymbol \cdot \vec H.
\end{equation}
Note that condition \eqref{pre:finie_total_curvature} ensures
\begin{equation} \label{pre:L2curvature}
	H \in L^2(\Sigma).
\end{equation}

\subsubsection{Singular points and Gauss-Bonnet Theorem of weak branched immersions}

First of all let us recall the following result first proved by  M\"uller-Sv\v{e}r\'ak  \cite{MR1366547}. For a different proof using H\'elein's moving frames technique \cite{MR1913803}, see \cite[Lemma A.5]{MR3276154}; see also \cite[Theorem 3.1]{MR2928715}) and \cite[Section 2.1]{MR3276119}.

\begin{proposition}\label{prop:SingPoints}
Let $\vec \Phi: \Sigma \to \mathbb R^3$ be a \emph{weak branched conformal immersion with finite total curvature} with singular points $b_1, \ldots, b_N \in \Sigma$. Let $\lambda \in L^{\infty}_{loc} (\Sigma \setminus \{b_{1}, \ldots, b_{N}\})$ be the conformal factor, i.e. $g=\vec \Phi^{*} g_{\mathbb R^{3}}=e^{2\lambda} g_{c_{0}}$.
\\Then  $\vec \Phi\in W^{2,2}(\Sigma)$ and  the conformal factor  $\lambda$ is an element of  $ L^{1}(\Sigma)$. 
\\Moreover,  for each singular point $b_{j}, j=1,\ldots,N,$ there exists  a strictly positive integer $n_{j}\in \mathbb N$ such that the following holds:
\begin{itemize}
\item For every $b_{j}$ there exists a local conformal chart $z$ centred at $\{b_{j}\}=\{z=0\}$ such that
$$\lambda(z)=(n_{j}-1) \log|z|+\omega (z)$$
for some $\omega\in C^{0}\cap W^{1,2}$. 
\item  The multiplicity of the immersion $\vec \Phi$ at $\vec \Phi(b_{j})$ is $n_{j}$. Moreover, if $n_{j}=1$, then $\vec \Phi$ is a conformal immersion of a neighbourhood of $b_{j}$.
\item The conformal factor $\lambda$ satisfies the following singular Liouville equation in distributional sense
\begin{equation}\label{eq:PDElamb}
-\Delta_{g_{c_{0}}} \lambda = K_{\vec \Phi} e^{2\lambda}-K_0-2\pi \sum_{j=1} ^{N} \left[(n_j-1) \delta_{b_j}\right],
\end{equation}
where  $\delta_{b_j}$ is the Dirac delta centred at $b_{j}$,  $K_{\vec \Phi}$ is the Gaussian curvature of $\vec \Phi$,  and $K_0\in \mathbb R$ is the (constant) curvature of $(\Sigma, g_{c_{0}})$. 
\end{itemize}
\end{proposition}

By integrating the singular Liouville equation \eqref{eq:PDElamb}, we obtain the Gauss-Bonnet Theorem for weak  branched immersions:
\begin{equation}\label{eq:GaussBonnetBranchedImm}
\int_{\Sigma} K_{\vec \Phi} d\mu_{g}= 2\pi \chi(\Sigma) + 2\pi   \sum_{j=1} ^{N} (n_j-1),
\end{equation}
where  $\chi(\Sigma)$ is the Euler Characteristic of $\Sigma$.
\\Note in particular that, once the topology of $\Sigma$ is fixed, the number of branch points counted with multiplicity is bounded by the Willmore energy:
\begin{equation}\label{eq:ControlBranchPoints}
2\pi \sum_{j=1} ^{N} (n_j-1)=\int_{\Sigma} K_{\vec \Phi} d\mu_{g} - 2\pi \chi(\Sigma) \leq 2  \int_{\Sigma} H^{2} d\mu_{g} - 2\pi \chi(\Sigma),
\end{equation}
Moreover, the Willmore energy controls the $L^{2}$ norm squared of the second fundamental form:
\begin{equation}\label{eq:ControlL2II}
 \int_{ \Sigma} | \vec{\mathbb I}|^{2} \, d\mu_{g} =  4 \int_{ \Sigma} H^{2} \, d\mu_{g}  - 2   \int_{ \Sigma}  K_{\vec \Phi} d\mu_{g}   \leq    4 \int_{ \Sigma} H^{2} \, d\mu_{g} -  4\pi \chi(\Sigma).
\end{equation}

\subsubsection{Simon's monotonicity formula and Li-Yau inequality for weak branched immersions}
Let $\vec \Phi \in \mathcal F_{\Sigma}$ be any weak branched conformal immersion with finite total curvature and branch points $\{b_1, \ldots, b_N\}$. In the usual way (by splitting the vector field in its tangential and normal parts and using integration by parts) one shows 
\begin{equation} \label{pre:variational_formula} 
	\int_{\Sigma} \Div_{\vec \Phi} \vec X \, d\mu_{\vec \Phi} = - 2 \int_{\Sigma} \vec X \boldsymbol{\cdot} \vec H_{\vec \Phi} \, d\mu_{\vec \Phi}
\end{equation} 
whenever $\vec X \in W^{1,2}(\Sigma, \mathbb R^3)$ has compact support in $\Sigma \setminus \{b_1, \ldots, b_N\}$, where in a local chart $x$,
\begin{equation*}
	\Div_{\vec \Phi} \vec X := g^{ij}\partial_{x^i} \vec X \boldsymbol{\cdot} \partial_{x^j} \vec \Phi. 
\end{equation*}
A simple cut-off argument together with \eqref{pre:L2curvature} shows that the first variation formula \eqref{pre:variational_formula} is true for all $\vec X \in  W^{1,2}(\Sigma, \mathbb R^3)$. In the following we will gather a couple of facts that are well known for weak unbranched immersions and, due to \eqref{pre:variational_formula}, are also valid for weak branched conformal immersions with finite total curvature. Firstly, letting $\vec X(p) := \vec \Phi(p) - \vec \Phi(a_0)$ for $p \in \Sigma$ and some fixed $a_0 \in \Sigma$, one has $\Div_{\vec \Phi} \vec X = 2$ and hence, see Simon \cite[Lemma 1.1]{MR857667}
\begin{equation} \label{pre:diameter_bound}
	\sqrt{\area \vec \Phi} \leqslant \diam \vec \Phi[\Sigma] \sqrt{\int_{\Sigma} H_{\vec \Phi}^2 \, d\mu_{\vec \Phi}}.
\end{equation} 
The push forward measure $\mu := \vec \Phi_\# \mu_{\vec \Phi}$ of $\mu_{\vec \Phi}$ defines a $2$-dimensional integral varifold in $\mathbb R^3$ with multiplicity function $\theta^2(\mu, x) = \mathscr H^0 (\vec \Phi^{-1} \{x\})$ (here $\mathscr H^0$ denotes the $0$-dimensional Hausdorff measure, i.e. the counting measure) and approximate tangent space $T_x\mu = d\vec \Phi [T_p\Sigma]$ almost everywhere when $x = \vec \Phi(p)$. See Simon \cite[Chapter 4]{MR756417} for an introduction on varifolds and Kuwert and Li \cite[Section 2.2]{MR2928715} for the context of weak unbranched immersions. From \eqref{pre:variational_formula} and the co-area formula (see for instance \cite[Equation 12.7]{MR756417}), the first variation formula for the varifold $\mu$ becomes
\begin{equation} \label{pre:first_variation}
	\int \Div_{\mu} \phi \, d\mu = - 2 \int \phi \boldsymbol{\cdot} H_\mu \, d\mu \qquad \text{for } \phi \in C^1_c(\mathbb R^3, \mathbb R^3)
\end{equation}
where the weak mean curvature is almost everywhere given by
\begin{equation*}
	H_{\mu}(x) = 
	\begin{cases} 
		\frac{1}{\theta^2(\mu,x)} \sum_{p \in \vec \Phi^{-1}(x)} \vec H_{\vec \Phi}(p) & \text{if } \theta^2(\mu,x) > 0 \\ 0 & \text{else}. 
	\end{cases}
\end{equation*}
The first variation formula \eqref{pre:first_variation} leads to Simon's monotonicity formula \cite[17.4]{MR756417} which implies (see for instance Rivi\`ere \cite[Section 5.3]{MR3524220} or Kuwert and Sch\"atzle \cite[Appendix]{MR2119722}) the Li-Yau inequality \cite[Theorem 6]{MR674407}
\begin{equation} \label{pre:Li-Yau-inequality}
	\theta^2(\mu, x) \leqslant \frac{1}{4\pi} \int_{\Sigma} H_{\vec \Phi}^2 \, d\mu_{\vec \Phi}. 
\end{equation}
Consequently,
\begin{equation} \label{pre:lower_bound_Willmore}
	\inf_{\vec{\Phi} \in \mathcal{F}_\Sigma} \int_{\Sigma} H_{\vec \Phi}^2 \, d\mu_{\vec \Phi} \geqslant 4\pi.
\end{equation}
Moreover, if $\vec{\Phi} \in \mathcal{F}_\Sigma$ with $\int_{\Sigma} H_{\vec \Phi}^2 \, d\mu_{\vec \Phi} < 8\pi$, then $\vec \Phi$ is an embedding (compare also with Proposition \ref{prop:SingPoints}).

\subsection{Canham-Helfrich energy}
Given real numbers $c_0 \in \mathbb R$ and $\alpha, \rho \geqslant 0$ as well as a weak branched conformal immersion with finite total curvature $\vec \Phi: \Sigma \to \mathbb R^3$, we define the Canham-Helfrich energy $\mathcal{H}^{c_0}_{\alpha, \rho}(\vec \Phi)$ in its most general form by 
\begin{equation} \label{pre:CH-energy}
\mathcal{H}^{c_0}_{\alpha, \rho}(\vec \Phi) := \int_\Sigma (H_{\vec \Phi} - c_0)^2 \, d\mu_{\vec \Phi} + \alpha \int_\Sigma 1 \, d\mu_{\vec \Phi} + \rho \int_\Sigma \vec n_{\vec \Phi} \boldsymbol \cdot \vec \Phi \, d\mu_{\vec \Phi}.
\end{equation}
Note that, in case $\vec \Phi: \Sigma \to \mathbb R^3$ is a smooth (actually Lipschitz is enough) embedding, by the Divergence Theorem the last integral equals the volume enclosed by $\vec \Phi(\Sigma)$. 
\\The parameter $\alpha$ is referred to as \emph{tensile stress}, $\rho$ as \emph{osmotic pressure}. Compare this definition for instance with \cite[Equation (3.6)]{MR3518329} or \cite{MR3757086}.

\section{Existence of minimisers}
\label{sec:existence}
In this chapter we will prove compactness of sequences with uniformly bounded Willmore energy and area as well as lower semi-continuity of the Canham-Helfrich energy under this convergence, see Theorem \ref{thm:compactness_and_lsc}. 
The proof of Theorem \ref{thm:compactness_and_lsc} will build on top of \cite{MR3276119}  and the next Lemma \ref{lem:convergence_outsidecon_centration_points}  which establishes the convergence of the constraints and the lower semi-continuity of the Willmore energy away from the branch points (Lemma  \ref{lem:convergence_outsidecon_centration_points} should be compared with \cite[Lemma 5.2]{MR3524220}). 
\begin{lemma}[Convergence outside the branch points]
	\label{lem:convergence_outsidecon_centration_points}
	Suppose $\vec \xi_1, \vec \xi_2,\ldots \in \mathcal F_{\mathbb S^2}$ is a sequence of weak branched conformal immersions with finite total curvature of the $2$-sphere~$\mathbb S^2$ into~$\mathbb R^3$, $\mu_1,\mu_2,\ldots$ are the corresponding Radon measures on $\mathbb S^2$, $\vec n_1, \vec n_2, \ldots$ are the corresponding Gauss maps, 
	\begin{equation} \label{cocp:uniformly_bounded_energy}
		\sup_{k \in \mathbb N}\int_{\mathbb S^2}|d \vec n_k|^2\,d\mu_k< \infty,
	\end{equation}
	there exists $\vec \xi_\infty \in \mathcal F_{\mathbb S^2}$, a positive integer $N$, and $b_1,\ldots,b_N \in \mathbb S^2$ such that
	\begin{gather}
		\label{cocp:uniformly_bounded_conformal_factor}
		\sup_{k \in \mathbb N} \|\log|d \vec \xi_k|\|_{L^\infty_{\mathrm{loc}}(\mathbb S^2\setminus\{b_1,\ldots,b_N\})} < \infty, \\ 
		\label{cocp:weak_convergence_hypothesis}
		\vec \xi_k \rightharpoonup \vec \xi_\infty \qquad \text{ as }k\to\infty\text{ weakly in } W^{2,2}_{\mathrm{loc}}(\mathbb S^2\setminus\{b_1,\ldots,b_N\},\mathbb R^3).
	\end{gather}
	Then, there exists a sequence of positive numbers $s_1,s_2,\ldots$ converging to zero such that 
	\begin{align}
		\label{cocp:convergence_area}
		\int_{\mathbb S^2} 1 \, d\mu_\infty & = \lim_{k\to\infty}\int_{\mathbb S^2\setminus \bigcup_{i=1}^NB_{s_k}(b_i)} 1 \, d\mu_k \\
		\label{cocp:convergence_mean_curvature}
		\int_{\mathbb S^2} H_{\infty}\,d\mu_\infty & = \lim_{k\to\infty}\int_{\mathbb S^2\setminus \bigcup_{i=1}^NB_{s_k}(b_i)}H_{k}\,d\mu_k \\ 
		\label{cocp:convergence_volume}
		\int_{\mathbb S^2} \vec n_ \infty \boldsymbol{\cdot} \vec \xi_ \infty \,d\mu_ \infty & = \lim_{k\to\infty} \int_{\mathbb S^2 \setminus \bigcup_{i=1}^NB_{s_k}(b_i)} \vec n_k \boldsymbol{\cdot} \vec \xi_k \,d\mu_k
	\end{align}
	where the balls are taken with respect to the geodesic distance on the standard $\mathbb S^2$, $\mu_ \infty$~and~$\vec n_\infty$ are the Radon measure and the Gauss map corresponding to $\vec \xi_ \infty$, and the $H_k$'s and $H_\infty$ are the mean curvatures corresponding to the $\vec \xi_k$'s and $\vec \xi_\infty$. Equations \eqref{cocp:convergence_area}--\eqref{cocp:convergence_volume} remain valid for $s_k$ replaced by any sequence~$t_k$ converging to zero and satisfying $t_k \geqslant s_k$, for all $k \in \mathbb N$. 
	
	Moreover, for any sequence $s_1,s_2, \ldots$ of positive numbers converging to zero, there exists a sequence $t_k \geqslant s_k$ converging to zero such that
	\begin{equation}
		\label{cocp:lsc-Willmore_energy}
		\int_{\mathbb S^2} H_{\infty}^2\,d\mu_\infty \leqslant \liminf_{k \to \infty}\int_{\mathbb S^2\setminus \bigcup_{i=1}^NB_{t_k}(b_i)} H_{k}^2\,d\mu_k. 
	\end{equation}
\end{lemma}

\begin{proof}
	Suppose $U$ is an open subset of $\mathbb S^2 \setminus \{b_1, \ldots, b_N\}$, $K$ is a compact subset of~$U$, and $x: U \to \mathbb R^2$ is a conformal chart for~$\mathbb S^2$. Denote by 
	\begin{equation*}
		\lambda_k = \log |\partial_{x^1} \vec \xi_k|, \qquad \lambda_\infty = \log |\partial_{x^1} \vec \xi_\infty| 
	\end{equation*}
	the conformal factors. Notice that the volume element corresponding to $\vec \xi_k$ is given by $e^{2\lambda_k}$. In a first step we will show that
	\begin{gather} \label{pcocp:convergence_area_element}
		e^{2\lambda_k} \to e^{2\lambda_\infty} \qquad \text{as } k \to \infty \text{ in }L^p(x[K]), \\ 
		\label{pcocp:convergence_volume}
		\vec n_k \boldsymbol{\cdot} \vec \xi_k e^{2\lambda_k} \to \vec n_ \infty \boldsymbol{\cdot} \vec \xi_ \infty e^{2\lambda_\infty} \qquad \text{as } k \to \infty \text{ in }L^p(x[K])
	\end{gather}
	for any $1 \leqslant p < \infty$, as well as
	\begin{align}
		\label{pcocp:convergence_mean_curvature}
		\int_{K} H_{\infty}\,d\mu_\infty & = \lim_{k\to\infty}\int_{K}H_{k}\,d\mu_k, \\ 
		\label{pcocp:lsc-Willmore_energy}
		\int_{K} H_{\infty}^2\,d\mu_\infty & \leqslant \liminf_{k \to \infty}\int_{K}H_{k}^2\,d\mu_k.
	\end{align}
	A simple argument by contradiction shows that it is enough to prove the statement after passing to a subsequence of~$k$. Since the $\vec \xi_k$'s and $\vec \xi_\infty$ are conformal and $x$ is a conformal chart, we can write the mean curvature vector as 
	\begin{equation*}
		2\vec{H}_{k} = e^{-2\lambda_k}\Delta \vec \xi_k, \qquad 2\vec{H}_{\infty} = e^{-2\lambda_\infty}\Delta \vec \xi_\infty  
	\end{equation*}
	where $\Delta$ is the flat Laplacian with respect to $x$. By Hypothesis \eqref{cocp:weak_convergence_hypothesis}, we have that 
	\begin{equation*} %\label{pcocp:weak_convergence_Laplacian}
		\vec{H}_{k} e^{2\lambda_k} = \frac{1}{2}\Delta \vec \xi_k \rightharpoonup \frac{1}{2}\Delta \vec \xi_\infty = \vec{H}_{\infty} e^{2\lambda_\infty}
	\end{equation*}
	as $k \to \infty$ weakly in $L^2(x[K],\mathbb R^3)$, which implies \eqref{pcocp:convergence_mean_curvature}. 
	
	By the Rellich-Kondrachov Compactness Theorem, 
	after passing to a subsequence, there holds
	\begin{equation} \label{pcocp:strong_convergence}
		\partial_{x^1} \vec \xi_k \to \partial_{x^1} \vec \xi_\infty \qquad \text{as } k\to\infty\text{ in } L^p_{\mathrm{loc}}(x[U],\mathbb R^3) 
	\end{equation}
	for any $1 \leqslant p < \infty$. Therefore, using Hypothesis \eqref{cocp:uniformly_bounded_conformal_factor} and passing to a further subsequence, it follows
	\begin{equation*}
		e^{-\lambda_k} = |\partial_{x^1} \vec \xi_k|^{-1} \to |\partial_{x^1} \vec \xi_\infty|^{-1} = e^{-\lambda_\infty} \qquad \text{as }k\to \infty\text{ in }L^2(x[K]).
	\end{equation*}
	It follows
	\begin{equation*}
		\vec{H}_{k}\sqrt{e^{2\lambda_k}} = \frac{1}{2}e^{-\lambda_k}\Delta \vec \xi_k \rightharpoonup \frac{1}{2}e^{-\lambda_\infty}\Delta \vec \xi_\infty  = \vec{H}_{\infty}\sqrt{e^{2\lambda_k}}
	\end{equation*}
	as $k\to\infty$ weakly in $L^2(x[K], \mathbb R^3)$, which implies \eqref{pcocp:lsc-Willmore_energy} by lower semi-continuity of the $L^2$-norm under weak convergence.
	
	Similarly, from Hypothesis \eqref{cocp:uniformly_bounded_conformal_factor} and the strong convergence \eqref{pcocp:strong_convergence}, we infer \eqref{pcocp:convergence_area_element}. 
	
	Again by the strong convergence \eqref{pcocp:strong_convergence} and Hypothesis \eqref{cocp:uniformly_bounded_conformal_factor}, we can extract a subsequence such that by dominated convergence,
	\begin{equation*}
		\vec n_k = e^{-2\lambda_k} (\partial_{x^1} \vec \xi_k \times \partial_{x^2} \vec \xi_k) \to e^{-2\lambda_\infty} (\partial_{x^1} \vec \xi_\infty \times \partial_{x^2} \vec \xi_\infty) = \vec n_\infty 
	\end{equation*} 
	as $k \to \infty$ in $ L^p(x[K], \mathbb R^3)$ for any $1\leqslant p < \infty$. Using this and the fact that by the Rellich-Kondrachov Compactness Theorem
	\begin{equation*}
		\vec \xi_k \to \vec \xi_\infty \qquad \text{as } k \to \infty \text{ in } L^p(x[K], \mathbb R^3)
	\end{equation*}
	for any $1 \leqslant p < \infty$, one verifies \eqref{pcocp:convergence_volume}.
	
	Next, let $r_k$ be any sequence of positive numbers converging to zero and abbreviate 
	\begin{equation*}
		f_k = e^{2\lambda_k}, \qquad f_\infty = e^{2\lambda_\infty}, \qquad K_{r_k} =\mathbb S^2 \setminus \bigcup_{i=1}^NB_{r_k}(b_i).
	\end{equation*} 
	First, notice that for any Borel function $f$ on~$\mathbb S^2$ with $\int_{\mathbb S^2}|f|\,d\mu_\infty < \infty$, there holds 
	\begin{equation} \label{pcocp:limit-lhs}
		\lim_{k\to\infty} \int_{K_{r_k}}f \, d\mu_{\infty} = \int_{\mathbb S^2} f \,d\mu_{\infty}
	\end{equation}
	which is a consequence of the dominated convergence theorem and the fact that finite sets have $\mu_\infty$~measure zero. Let $n_0 = 1$. For each positive integer $j$, we use \eqref{pcocp:convergence_area_element} to inductively choose $n_j > n_{j-1}$ such that  
	\begin{equation*}
		\int_{x[K_{r_j}]}|f_k - f_\infty|\,d\mathscr L^2 \leqslant \frac{1}{j} \qquad \text{for all } k\geqslant n_j.
	\end{equation*}
	Moreover, define $l_k = j$ for all integers $k$ with $n_{j-1}< k \leqslant n_j$ and define $s_k = r_{l_k}$. Then, we have that $s_k \to 0$ as $k \to \infty$ as well as
	\begin{equation} \label{pcocp:L^1-convergence_area}
		\int_{x[K_{s_k}]} |f_k - f_\infty| \, d\mathscr L^2 \to 0 \qquad \text{as } k \to \infty
	\end{equation}
	which in particular remains valid for $s_k$ replaced by any $t_k \geqslant s_k$. Hence, by \eqref{pcocp:limit-lhs} we can deduce \eqref{cocp:convergence_area}.
	Using the convergence on compact sets \eqref{pcocp:convergence_volume}--\eqref{pcocp:lsc-Willmore_energy}, Equations \eqref{cocp:convergence_mean_curvature}--\eqref{cocp:lsc-Willmore_energy} follow similarly. It only remains to show that Equation \eqref{cocp:convergence_mean_curvature} is still valid after replacing $s_k$ by any sequence $t_k \geqslant s_k$ converging to zero. Hence, we only have to show that 
	\begin{equation*}
		\int_{K_{s_k} \setminus K_{t_k}} H_k \, d\mu_k \to 0 \qquad \text{as } k \to \infty.
	\end{equation*} 
	This follows as by H\"older's inequality
	\begin{equation*}
		\Bigl| \int_{K_{s_k} \setminus K_{t_k}} H_k \, d\mu_k \Bigr| \leqslant \Bigl(\int_{\mathbb S^2} H_k^2 \, d\mu_k \Bigr)^{1/2} \Bigl(\int_{K_{s_k} \setminus K_{t_k}} 1 \, d\mu_k \Bigr)^{1/2}.
	\end{equation*}
	The first factor on the right hand side is bounded by \eqref{cocp:uniformly_bounded_energy}. To see that the second factor goes to zero as $k$ tends to infinity, we apply \eqref{pcocp:L^1-convergence_area} and the fact that $\mu_{\infty}(\bigcup_{i = 1}^NB_{t_k}(b_i)) \to 0$ as $k \to \infty$.
\end{proof}

In the following we will define the notion of a bubble tree. The idea is that the different bubbles can be parametrised by decomposing a single $2$-sphere. The bubbles can then be attached to each other by a Lipschitz map, see \eqref{bt:parametrisation} and \eqref{bt:branch_point}.

\begin{definition}[Bubble tree of weak immersions, see \protect{\cite[Definition 7.1]{MR3276119}}] \label{def:bubble_tree}
	An $N+1$ tuple $\vec T = (\vec f, \vec \Phi^1, \ldots, \vec \Phi^N)$ is called a \emph{bubble tree of weak immersions} if and only if $N$ is a positive integer, $\vec f \in W^{1, \infty}(\mathbb S^2, \mathbb R^3)$, and $\vec \Phi^1, \ldots, \vec \Phi^N \in \mathcal F_{\mathbb S^2}$ are weak branched conformal immersions with finite total curvature such that the following holds. 
	
	There exist open geodesic balls $B^1, \ldots, B^N \subset \mathbb S^2$ such that 
	\begin{itemize}
		\item $\overline{B^1} = \mathbb S^2$ and for all $i \neq i'$ either $\overline{B^i} \subset B^{i'}$ or $\overline{B^{i'}} \subset B^i$. 
	\end{itemize}
	For all $i \in \{1, \ldots, N\}$ there exists a positive integer $N^i$ and disjoint open geodesic balls $B^{i,1}, \ldots, B^{i,N^i} \subset \mathbb S^2$ whose closures are included in $B^i$ such that
	\begin{itemize}
		\item for all $i' \neq i$ either $\overline{B^i} \subset B^{i'}$ or $\overline{B^{i'}} \subset B^{i,j}$ for some $j \in \{1, \ldots, N^i\}$.
	\end{itemize}
	For all $i \in \{1, \ldots, N\}$ there exist distinct points $b^{i,1}, \ldots, b^{i,N^i} \in \mathbb S^2$ and a Lipschitz diffeomorphism 
	\begin{equation*}
		\Xi^i: B^i \setminus \bigcup_{j = 1}^{N^i - 1} \overline{B^{i,j}} \to \mathbb S^2 \setminus \{b^{i,1}, \ldots, b^{i, N^i}\}
	\end{equation*} 
	which extends to a Lipschitz map
	\begin{equation*}
		\overline \Xi_i: \overline{B^i} \setminus \bigcup_{j=1}^{N^i - 1}B^{i,j} \to \mathbb S^2
	\end{equation*}
	such that
	\begin{equation*}
		\overline{\Xi}_i[\partial B^{i,j}] = b^{i,j} \text{ whenever } j \in \{1, \ldots, N^{i}-1\}, \qquad \overline \Xi_i[\partial B^i] = b^{i,N^i}.
	\end{equation*}
	Moreover, for all $i \in \{1, \ldots, N\}$,
	\begin{equation} \label{bt:parametrisation}
		\vec f(x) = (\vec \Phi^i \circ \Xi^i)(x) \qquad \text{whenever } x \in B^i \setminus \bigcup_{j = 1}^{N^{i}-1} \overline{B^{i,j}}
	\end{equation}
	and for all $j \in \{1, \ldots, N^i\}$ there exists $p^{i,j} \in \mathbb R^3$ such that 
	\begin{equation} \label{bt:branch_point}
		\vec f(x) = p^{i,j} \qquad \text{whenever } x \in B^{i,j} \setminus \bigcup_{i' \in J^{i,j}} \overline{B^{i'}}
	\end{equation}
	where $J^{i,j} = \{i': \overline{B^{i'}} \subset B^{i,j}\}$.
	
	Finally, we define
	\begin{gather*}
		\mathcal W(\vec T) := \sum_{i = 1}^N \int_{\mathbb S^2} H_{\vec \Phi^i}^2 \, d\mu_{\vec \Phi^i}, \qquad \area(\vec T) := \sum_{i = 1}^N \int_{\mathbb S^2} 1 \, d\mu_{\vec \Phi^i}, \\
		\vol(\vec T) := \sum_{i = 1}^N \int_{\mathbb S^2} \vec n_{\vec\Phi^i} \boldsymbol{\cdot} \vec \Phi^i \, d\mu_{\vec \Phi^i}. 
	\end{gather*}
\end{definition}

The next theorem establishes the weak closure of bubble trees, as well as the  convergence of  the constraints in the Helfrich problem and the lower semi-continuity of the Willmore energy.  The proof builds on top of \cite{MR3276119}.
\begin{theorem}[Weak closure and lower semi-continuity of bubble trees] \label{thm:compactness_and_lsc}
	Suppose $\vec T_k = (\vec f_k, \vec \Phi^1_k, \ldots, \vec \Phi^{N_k}_k)$ is a sequence of bubble trees of weak immersions and
	\begin{equation} \label{wc:uniformly_bounded_energy}
		\limsup_{k \to \infty} \sum_{i = 1}^{N_k} \int_{ \mathbb S^2 } 1 + |d\vec n_{\vec \Phi_k^i}|^2 \, d\mu_{\vec \Phi_k^i} < \infty, \qquad \liminf_{k \to \infty} \sum_{i = 1}^{N_k} \diam \vec \Phi_k^i[\mathbb S^2] > 0.
	\end{equation}
	
	Then, there exists a subsequence of $\vec T_k$ which we again denote by $\vec T_k$ such that $N_k = N$ for some positive integer $N$ and there exists a sequence of diffeomorphisms $\Psi_k$ of $\mathbb S^2$ such that
	\begin{gather*} 
		\vec f_k \circ \Psi_k \to \vec u_\infty \qquad \text{ as } k \to \infty \text{ uniformly in } C^0(\mathbb S^2, \mathbb R^3), \\
		\area \vec f_k[\mathbb S^2] \to \area \vec u_\infty[\mathbb S^2] \qquad \text{as } k \to \infty 
	\end{gather*}
	for some $\vec u_\infty \in W^{1, \infty}(\mathbb S^2, \mathbb R^3)$. Moreover, for all $i \in \{1,\ldots, N\}$ there exists a positive integer $Q^i$ and sequences $f_k^{i,1}, \ldots, f_k^{i,Q^i}$ of positive conformal diffeomorphisms of $\mathbb S^2$ such that for each $j \in \{1, \ldots, Q^i\}$ there exist finitely many points $b^{i,j,1}, \ldots b^{i,j,Q^{i,j}} \in \mathbb S^2$ with
	\begin{equation} \label{wc:weak_convergence}
		\vec \Phi^i_k \circ f_k^{i,j} \rightharpoonup \vec \xi^{i,j}_\infty \qquad \text{ as } k \to \infty \text{ weakly in } W_{\mathrm{loc}}^{2,2}(\mathbb S^2 \setminus \{b^{i,j,1}, \ldots b^{i,j,Q^{i,j}}\}, \mathbb R^3)
	\end{equation}
	for some branched Lipschitz conformal immersion $\vec \xi_\infty^{i,j} \in \mathcal{F}_{\mathbb S^2}$. Furthermore, 
	\begin{equation*}
		\vec T_\infty := \bigl(\vec u_\infty,(\vec \xi^{1,j}_\infty)_{j = 1, \ldots, Q^1}, \ldots, (\vec \xi_\infty^{N,j})_{j = 1, \ldots, Q^N}\bigr)
	\end{equation*}
	is a bubble tree of weak immersions and
	\begin{gather*}
		\mathcal W(\vec T_\infty) \leqslant \liminf_{k \to \infty} \mathcal W (\vec T_k), \quad \area (\vec T_\infty) = \lim_{k\to\infty} \area(\vec T_k), \quad \vol(\vec T_\infty) = \lim_{k\to\infty} \vol(\vec T_k)
	\end{gather*}
	as well as 
	\begin{equation*}
		\sum_{i = 1}^{N} \sum_{j = 1}^{Q^i} \int_{ \mathbb S^2 } H_{\vec \xi_\infty^{i,j}} \, d \mu_{\vec \xi_\infty^{i,j}} = \lim_{k\to\infty} \sum_{i = 1}^N \int_{ \mathbb S^2 } H_{\vec \Phi_k^i} \, d\mu_{\vec \Phi_k^i}. 
	\end{equation*}
\end{theorem}

\begin{proof}
	We first consider the special case where $N_k = 1$ for all positive integers $k$. By \cite[Theorem 1.5]{MR3276119}, it then only remains to show the convergence properties of the Willmore energy~$\mathcal W$, the volume, and the integral of the mean curvature. In view of Lemma \ref{lem:convergence_outsidecon_centration_points}, we can add Equations \eqref{cocp:convergence_mean_curvature}--\eqref{cocp:lsc-Willmore_energy} for $\vec \xi_k$ replaced by $\vec \Phi_k \circ f_k^i$ to the conclusion of the Domain Decomposition Lemma \cite[Theorem 6.1]{MR3276119}. Therefore, adapting the proof of \cite[Theorem 1.5]{MR3276119}, we get the following statement. 
	 
	After passing to a subsequence and denoting $\vec \Phi_k := \vec \Phi_k^1$, there exists a positive integer $N$, sequences $f_k^1, \ldots, f_k^N$ of positive conformal diffeomorphisms of $\mathbb S^2$, and for each $i \in \{1, \ldots, N\}$ there exist points $b^{i,1}, \ldots, b^{i,N^i} \in \mathbb S^2$ such that \eqref{wc:weak_convergence} and \eqref{intro:continuous_convergence} hold. Moreover, there exists a sequence of positive numbers $s_k$ converging to zero such that for $i = 1, \ldots, N$ Equations \eqref{cocp:convergence_area}--\eqref{cocp:lsc-Willmore_energy} are satisfied for $\vec \xi_k$ replaced by $\vec \Phi_k \circ f^i_k$. Furthermore, defining
	\begin{equation*}
		S^i_k := \mathbb S^2 \setminus \bigcup_{j=1}^{N^i} B_{s_k}(b^{i,j}),
	\end{equation*} 
	and for $j = 1, \ldots, N^i$ the sets of indices 
	\begin{equation*}
		J^{i,j} := \{i': \forall k \in \mathbb N: \bigl((f_k^i)^{-1} \circ f_k^{i'}\bigr)[S_k^{i'}]\subset B_{s_k}(b^{i,j})\},
	\end{equation*}
	and 
	\begin{equation*}
		\hat{J}^{i,j}:= \{i' \in J^{i,j}: \forall k \in \mathbb N: \nexists i'':\bigl((f_k^i)^{-1} \circ f_k^{i'}\bigr)[S_k^{i'}]\subset \conv \bigl((f_k^i)^{-1} \circ f_k^{i''}\bigr)[S_k^{i''}]\},
	\end{equation*}
	and the necks 
	\begin{equation*}
		S_k^{i,j} := B_{s_k}(b^{i,j}) \setminus \bigcup_{i' \in \hat J^{i,j}}\bigl((f^i_k)^{-1} \circ f^{i'}_k\bigr)\bigl[\mathbb S^2 \setminus B_{s_k}(b^{i',N^{i'}})\bigr],
	\end{equation*}
	there holds
	\begin{equation} \label{pwc:neck_condition}
		\lim_{k\to\infty} \int_{S_k^{i,j}} 1 \, d\mu_{\vec \Phi_k \circ f^i_k} = 0, \qquad \lim_{k\to\infty} \diam (\vec \Phi_k \circ f^i_k)[S^{i,j}_k] = 0.
	\end{equation}
	Finally, for any $\mu_{\vec \Phi_k}$ integrable Borel function $\varphi$ on $\mathbb S^2$, we get
	\begin{equation} \label{pwc:integral_identity}
		\int_{ \mathbb S^2 } \varphi \, d\mu_{\vec\Phi_k} = \sum_{i = 1}^N \int_{\mathbb S^2\setminus\bigcup_{j=1}^{N_i}B_{s_k}(b^{i,j})}\varphi \circ f^i_k\,d\mu_{\vec \Phi_k\circ f^i_k} + \sum_{i = 1}^N\sum_{j = 1}^{N^i-1}\int_{S_k^{i,j}}\varphi\circ f^i_k\,d\mu_{\vec \Phi_k\circ f^i_k}.
	\end{equation}
	We notice that by the strong convergence \eqref{intro:continuous_convergence},  
	\begin{equation*}
	\sup_{k \in \mathbb N} \sup_{S_k^{i,j}}|\vec \Phi_k \circ f^i_k| \leqslant C < \infty
	\end{equation*}
	for some finite number $C > 0$. Hence, by H\"older's inequality 
	\begin{gather*}
		\Bigl|\int_{S_k^{i,j}}H_{\vec \Phi_k\circ f^i_k}\,d\mu_{\vec \Phi_k\circ f^i_k}\Bigr| \leqslant \Bigl(\int_{S_k^{i,j}}1\,d\mu_{\vec \Phi_k\circ f^i_k}\Bigr)^{1/2}\Bigl(\int_{\mathbb S^2}H_{\vec \Phi_k \circ f_k^i}^2\,d\mu_{\vec \Phi_k \circ f_k^i}\Bigr)^{1/2}, \\
		\Bigl|\int_{S_k^{i,j}}\vec n_{\vec \Phi_k\circ f^i_k} \boldsymbol{\cdot} (\vec \Phi_k\circ f^i_k)\,d\mu_{\vec \Phi_k\circ f^i_k}\Bigr| \leqslant \Bigl(
		\int_{S_k^{i,j}}1\,d\mu_{\vec \Phi_k\circ f^i_k} \Bigr)^{1/2} \Bigl(\int_{S_k^{i,j}}C^2\,d\mu_{\vec \Phi_k\circ f^i_k} \Bigr)^{1/2}.
	\end{gather*}
	By  \eqref{wc:uniformly_bounded_energy} and \eqref{pwc:neck_condition}, the right hand side of each line goes to zero as $k$ tends to infinity. That means the last term of Equation \eqref{pwc:integral_identity} goes to zero as $k$ tends to infinity when $\varphi$ is replaced by $H_{\vec \Phi_k}$ as well as when $\varphi$ is replaced by $\vec n_{\vec \Phi_k} \boldsymbol \cdot \vec \Phi_k$. Therefore, using \eqref{cocp:convergence_mean_curvature} and \eqref{cocp:convergence_volume}, we can conclude the convergence of the integrated mean curvature and the convergence of the volume from \eqref{pwc:integral_identity}. Similarly, we can conclude the lower semi-continuity of the Willmore energy $\mathcal W$ from \eqref{pwc:integral_identity} by replacing $\varphi$ with $H_{\vec \Phi_k}^2$, using super linearity of the limit inferior and by ignoring the non-negative second term in \eqref{pwc:integral_identity}.
	
	Now, the general case follows analogously to the proof of \cite[Theorem 7.2]{MR3276119}.
\end{proof}

\section{Regularity of minimisers}
\label{sec:regularity}

Throughout this section, $\Sigma$ denotes a smooth, oriented, and closed $2$-dimensional manifold. Moreover, $c_0$, $\alpha$, and $\rho$ are the parameters of the Canham-Helfrich energy, i.e. $c_0 \in \mathbb R$ and $\alpha, \rho \geqslant 0$, see \eqref{pre:CH-energy}. A (possibly branched) weak immersion $\vec \Phi \in \mathcal F_\Sigma$ with branch points $\{b_{1}, \ldots, b_{N}\}$ is called \emph{weak Canham-Helfrich immersion} if
\begin{equation} \label{CHimmersion}
	\left.\frac{d}{dt}\right|_{t = 0} \mathcal H^{c_0}_{\alpha,\rho}(\vec \Phi + t\vec \omega) = 0
\end{equation}
for all $\vec \omega \in C_{c}^\infty(\Sigma \setminus \{b_{1}, \ldots, b_{N}\}, \mathbb R^3)$.

In the following, we will first compute the Canham-Helfrich equation in divergence form, see Lemma \ref{lem:Canham-Helfrich_equation}. Then, we will prove that a weak immersion satisfying the Canham-Helfrich equation is smooth away from its branch points, see Theorem \ref{thm:smoothness}. The proof is based on the regularity theory for weak Willmore immersions developed by Rivi\`ere \cite{MR2430975, MR3524220}. An important step in Riviere's regularity theory is the discovery of hidden conservation laws for weak Willmore immersions. In the framework of Canham-Helfich immersions, the corresponding hidden conservation laws were discovered by Bernard \cite{MR3518329}.

\begin{lemma}[Canham-Helfrich Euler-Lagrange equation in divergence form] \label{lem:Canham-Helfrich_equation}
	Suppose $\vec \Phi \in \mathcal F_\Sigma$ is a weak Canham-Helfrich immersion with branch points $\{b_1, \ldots, b_N\}$. Then, away from its branch points, i.e. in conformal parametrisations from the open unit disk $D^2$ into a subset of $\Sigma \setminus \bigcup_{i = 1}^N B_\varepsilon (b_i)$ for any $\varepsilon > 0$, there holds 
	\begin{equation} \label{CHEL}
		\vec{\mathcal{W}} = -\Div \left[c_0 \nabla \vec n + (2c_0H - c_0^2 - \alpha) \nabla \vec \Phi - \frac{\rho}{2} \vec \Phi \times \nabla^\bot \vec \Phi \right] 
	\end{equation} 
	in $\mathcal D'(D^2, \mathbb R^3)$, where 
	\begin{equation} \label{Willmore_equation}
		\vec{\mathcal W} := \Div \frac{1}{2}\left[2\nabla \vec H - 3H\nabla \vec n + \vec H \times \nabla^\bot \vec n\right]
	\end{equation}
	corresponds to the first variation of the Willmore energy. 
\end{lemma}

\begin{proof}
	After composing with a conformal chart away from the branch points, we may assume that $\vec \Phi$ is a map $D^2 \to \mathbb R^3$. Let $\vec \omega \in C^\infty_{c}(D, \mathbb R^3)$	and define $\vec \Phi_t := \vec \Phi + t\vec \omega$ for $t \in \mathbb R$. The conformal factor $\lambda$ is given by $2e^{2\lambda} = |\nabla \vec \Phi|^2$ and the metric coefficients $(g_t)_{ij}$ by $(g_t)_{ij} = \partial_i \vec \Phi_t \boldsymbol{\cdot} \partial_j \vec \Phi_t$. Standard computations (see for instance \cite[(7.8)--(7.10)]{MR3524220}) give 
	\begin{align*} 
		&\Bigl.\frac{d}{dt}\Bigr|_{t = 0} (g_t)^{ij} = - e^{-4\lambda}(\partial_i \vec \omega \boldsymbol{\cdot} \partial_j \vec \Phi + \partial_i \vec \Phi \boldsymbol{\cdot} \partial_j \vec \omega) \\
		&g^{ij} \Bigl(\partial_i \Bigl.\frac{d}{dt}\Bigr|_{t = 0} \vec n_t \boldsymbol{\cdot} \partial_j \vec \Phi \Bigr) = -e^{-2\lambda}\bigl(\partial_1(\partial_1 \vec \omega \boldsymbol{\cdot} \vec n) + \partial_2(\partial_2 \vec \omega \boldsymbol{\cdot} \vec n)\bigr) \\ 
		&\Bigl.\frac{d}{dt}\Bigr|_{t = 0}\sqrt{\det (g_t)_{ij}} = \partial_1 \vec \Phi \boldsymbol{\cdot} \partial_1 \vec \omega + \partial_2 \vec \Phi \boldsymbol{\cdot} \partial_2 \vec \omega. 
	\end{align*}
	Therefore, using
	\begin{align*}
		\Bigl.\frac{d}{dt}\Bigr|_{t = 0} H_t & = - \Bigl.\frac{d}{dt}\Bigr|_{t = 0} \frac{1}{2} (g_t)^{ij} \Bigl(\partial_i \vec n_t \boldsymbol{\cdot} \partial_j \vec \Phi_t \Bigr)\\
		& = -\frac{1}{2} \Bigl(\Bigl.\frac{d}{dt}\Bigr|_{t = 0} (g_t)^{ij} \Bigr) \partial_i \vec n \boldsymbol{\cdot} \partial_j \vec \Phi - g^{ij}\frac{1}{2}\Bigl(\partial_i \Bigl.\frac{d}{dt}\Bigr|_{t = 0} \vec n_t \boldsymbol{\cdot} \partial_j \vec \Phi + \partial_i \vec n \boldsymbol{\cdot} \partial_j \vec \omega \Bigr), 
	\end{align*}
	we obtain
	\begin{align}
		\Bigl.\frac{d}{dt}\Bigr|_{t = 0} \int_{D^2} H_t \, d\mu_t & = \int_{D^2} \Bigl(\Bigl.\frac{d}{dt}\Bigr|_{t = 0} H_t \Bigr) \sqrt{\det g_{ij}} + H \Bigl.\frac{d}{dt}\Bigr|_{t = 0}\sqrt{\det (g_t)_{ij}} \, d\mathscr L^2  \nonumber \\
		& = \int_{D^2} \frac{1}{2} e^{-2\lambda}\sum_{i,j = 1}^2(\partial_i \vec \omega \boldsymbol{\cdot} \partial_j \vec \Phi + \partial_i \vec \Phi \boldsymbol{\cdot} \partial_j \vec \omega) (\partial_{i} \vec n \boldsymbol \cdot \partial_j \vec \Phi) \nonumber \\
		& \qquad \quad + \frac{1}{2}\big(\partial_1(\partial_1 \vec \omega \boldsymbol \cdot \vec n) + \partial_2 (\partial_2 \vec \omega \boldsymbol \cdot \vec n) \bigr) - \frac{1}{2} g^{ij} \partial_i \vec n \boldsymbol{\cdot} \partial_j \vec \omega e^{2\lambda}  \nonumber \\
		& \qquad  \quad + H\bigl( \partial_1 \vec \Phi \boldsymbol \cdot \partial_1 \vec \omega + \partial_2 \vec \Phi \boldsymbol \cdot \partial_2 \vec \omega \bigr)\, d\mathscr L^2.	\label{eq:dintHPrePre}
	\end{align}
	Using that $\vec \omega$ has compact support in $D^2$,
	\begin{gather*}
		g^{ij} \partial_i \vec n \boldsymbol{\cdot} \partial_j \vec \omega e^{2\lambda} = e^{-2\lambda} \sum_{i,j = 1}^2 (\partial_i \vec n \boldsymbol{\cdot} \partial_j \vec \Phi)(\partial_j \vec \Phi \boldsymbol{\cdot} \partial_i \vec \omega),
	\end{gather*}
	and using the symmetry of the second fundamental form, i.e. $\partial_i \vec n \boldsymbol{\cdot} \partial_j \vec \Phi = \partial_j \vec n \boldsymbol{\cdot} \partial_i \vec \Phi$, we compute further 
	\begin{align}
		\Bigl.\frac{d}{dt}\Bigr|_{t = 0} \int_{D^2} H_t \, d\mu_t & = \int_{D^2} -\frac{1}{2} \vec \omega \boldsymbol \cdot \partial_1 \Bigl[e^{-2\lambda}\bigl((\partial_1 \vec n \boldsymbol \cdot \partial_1 \vec \Phi) \partial_1 \vec \Phi + (\partial_1 \vec n \boldsymbol \cdot \partial_2 \vec \Phi) \partial_2 \vec \Phi)\bigr)\Bigr] \nonumber  \\
		& \qquad \qquad - \frac{1}{2} \vec \omega \boldsymbol \cdot \partial_2 \Bigl[ e^{-2\lambda}\bigl((\partial_2 \vec n \boldsymbol \cdot \partial_1 \vec \Phi) \partial_1 \vec \Phi + (\partial_2 \vec n \boldsymbol \cdot \partial_2 \vec \Phi) \partial_2 \vec \Phi)\bigr) \Bigr]\nonumber  \\ 
		& \qquad \qquad - \vec \omega \boldsymbol \cdot \bigl(\partial_1(H \partial_1 \vec \Phi) + \partial_2(H \partial_2 \vec \Phi)\bigr) \, d\mathscr L^2 \nonumber  \\
		& = - \int_{D^2} \vec \omega \boldsymbol \cdot  \partial_1 \left(\frac{1}{2} \pi_T(\partial_1 \vec n) + H \partial_1 \vec \Phi\right) \nonumber  \\
		& \qquad \qquad + \vec \omega \boldsymbol \cdot \partial_2 \left(\frac{1}{2}\pi_T(\partial_2 \vec n) + H \partial_2 \vec \Phi\right) \, d\mathscr L^2 \label{eq:dintHPre}
	\end{align}
	which gives
	\begin{equation} \label{pCHEL:curvature_variation}  
		\Bigl.\frac{d}{dt}\Bigr|_{t = 0} \int_{D^2} H_t \, d\mu_t = -\int_{D^2}  \vec \omega \boldsymbol \cdot \Div\left[\frac{1}{2}\nabla \vec n + H \nabla \vec \Phi\right] \, d\mathscr L^2.
	\end{equation} 

	From \cite[Corollary 7.3]{MR3524220} we know 
	\begin{equation} \label{pCHEL:Willmore_variation}
	\Bigl.\frac{d}{dt}\Bigr|_{t = 0} \int_{D^2} H_t^2 \, d\mu_t = \int_{D^2} \vec \omega \boldsymbol \cdot \Div \frac{1}{2}\left[2\nabla \vec H - 3H\nabla \vec n + \vec H \times \nabla^\bot \vec n\right] \, d\mathscr L^2
	\end{equation}  
	see also \cite{MR2430975}. Moreover (see for instance \cite[Chapter 3.3]{MR3518329}) 
	\begin{equation} \label{pCHEL:area_variation}
		\Bigl.\frac{d}{dt}\Bigr|_{t = 0} \int_{D^2} 1 \, d\mu_t = - \int_{D^2} \vec \omega \boldsymbol \cdot \Div \nabla \vec \Phi \, d \mathscr L^2
	\end{equation}
	and
	\begin{equation} \label{pCHEL:volume}
		\Bigl.\frac{d}{dt}\Bigr|_{t = 0} \int_{D^2}  \vec n_t \boldsymbol \cdot \vec \Phi_t \, d\mu_t = - \int_{D^2} \vec \omega \boldsymbol \cdot \Div\left[\frac{1}{2} \vec \Phi \times \nabla^\bot \vec \Phi\right] \, d\mathscr L^2.
	\end{equation}
	Putting \eqref{pCHEL:curvature_variation} -- \eqref{pCHEL:volume} into \eqref{CHimmersion} yields \eqref{CHEL}.
\end{proof}

\begin{remark}
Usually in the literature  (see for instance \cite[Chapter 3.3]{MR3518329})  one finds the expression of the first variation for $\int H d\mu_{g}$ written as
\begin{equation} \label{eq:FirstVarIntHLit}
		\Bigl.\frac{d}{dt}\Bigr|_{t = 0} \int_{D^2} H_t \, d\mu_t = \int_{D^2} (\vec \omega  \boldsymbol \cdot \vec n) \left( \frac{1}{2} {\mathbb I}^{i}_{j} {\mathbb I}^{j}_{i}-2H^{2} \right) d\mu_{g}.
\end{equation}
It is not hard to check the equivalence of \eqref{eq:FirstVarIntHLit} with \eqref{pCHEL:curvature_variation} proved above. The advantage of the expression \eqref{pCHEL:curvature_variation} is two fold: first it invokes less regularity of the immersion map $\vec \Phi$, second it is already in divergence form. Both advantages will be useful in establishing the regularity of weak Canham-Helfrich immersions: indeed, \eqref{eq:FirstVarIntHLit} would correspond to an $L^{1}$ term in the Euler-Lagrange equation (which is usually a problematic right hand side for elliptic regularity theory) while \eqref{pCHEL:curvature_variation} corresponds to the divergence of an $L^{2}$ term (which is a much better right hand side in elliptic regularity).
\end{remark}

\begin{theorem}[Smoothness of weak Canham-Helfrich immersions] \label{thm:smoothness}
	Suppose $\vec \Phi \in \mathcal F_\Sigma$ is a weak Canham-Helfrich immersion. Then $\vec \Phi$ is a  $C^\infty$ immersion away from the branch points.
\end{theorem}

\begin{proof}
	After composing with a conformal chart of $\Sigma$ away from the branch points onto the unit disk $D^2$, we may assume that $\vec \Phi$ is a map $D^2 \to \mathbb R^3$ without branch points and $\vec \Phi$ satisfies the Canham-Helfrich equation \eqref{CHEL}. It is enough to show that $\vec \Phi \in C^\infty(B_{1/2}(0))$. The proof splits into three parts. \\
	
	\emph{Step 1: Conservation laws.} In view of the Canham-Helfrich equation \eqref{CHEL}, we define \footnote{Note that here, $\vec T$ is not a bubble tree.} $\vec T \in L^2(D^2, (\mathbb R^3)^2)$ by letting
	\begin{equation*}
		\vec T : = c_0 \nabla \vec n 
		+ (2c_0H - c_0^2 - \alpha) \nabla \vec \Phi - \frac{\rho}{2} \vec \Phi \times \nabla^\bot \vec \Phi. 
	\end{equation*}
	Then, $\Div \vec T = - \vec{\mathcal W}$ where $\vec{\mathcal W}$ is as in \eqref{Willmore_equation}. Hence $\vec{\mathcal W} \in H^{-1}(D^2, \mathbb R^3)$ and there exists a solution~$\vec V$ of 
	\begin{equation*}
	\begin{cases}
		\Delta \vec V = - \vec{\mathcal W} \\
		\vec V \in H^1_0(D^2, \mathbb R^3).
	\end{cases} 
	\end{equation*}
	Therefore, we can find 
	\begin{equation} \label{psCH:regularityXY}
		\vec X \in W^{2,2}(D^2,\mathbb R^3), \qquad Y \in W^{2,2}(D^2, \mathbb R)
	\end{equation}
	such that
	\begin{empheq}[left=\empheqlbrace]{align*} 
		\Delta \vec X & = \nabla \vec V \times \nabla \vec \Phi && \text{in } D^2 \\
		\vec X & = 0 && \text{on } \partial D^2  
	\end{empheq}
	and
	\begin{empheq}[left=\empheqlbrace]{align*} 
		\Delta Y & = \nabla \vec V \boldsymbol \cdot \nabla \vec \Phi && \text{in } D^2 \\
		Y & = 0 && \text{on } \partial D^2.  
	\end{empheq}
	After the breakthrough of Rivi\`ere \cite{MR2430975}, Bernard \cite[Chapter 2.2]{MR3518329} showed that by invariance of the Willmore functional under conformal transformation and the weak Poincar\'e Lemma, one can find potentials $\vec L, \vec R \in W^{1, 2}(D^2, \mathbb R^3)$, and $S \in W^{1,2}(D^2, \mathbb R)$ such that
	\begin{equation*}
	\left\{
		\begin{split}
		\nabla^\bot \vec L & = \vec T - \nabla \vec V \\
		\nabla^\bot \vec R & = \vec L \times \nabla^\bot \vec \Phi - \vec H \times \nabla \vec \Phi - \nabla \vec X \\
		\nabla^\bot S & = \langle \vec L, \nabla^\bot \vec \Phi \rangle - \nabla Y.
		\end{split}
	\right.	
	\end{equation*}
	Indeed, from \cite[Chapter 3.3]{MR3518329} we find that $\vec R, S, \vec X, Y$, and $\vec \Phi$ satisfy the following system of conservation laws 
   \begin{empheq}[left=\empheqlbrace]{align} \label{psCH:conservationR}
			\Delta \vec R & = \langle \nabla^\bot \vec n, \nabla S \rangle + \nabla^\bot \vec n \times \nabla \vec R + \Div \Bigl[\langle \vec n, \nabla Y\rangle + \frac{\rho}{4}|\vec \Phi|^2 \nabla \vec \Phi \Bigr] \\ \label{psCH:conservationS}
			\Delta S & =  \nabla^\bot \vec n \boldsymbol \cdot  \nabla \vec R \\
			\label{psCH:conservationY}
			\Delta Y & = |\nabla \vec \Phi|^2\bigl(-(c_0^2 + \alpha) +c_0H + \frac{\rho}{2} \vec \Phi \boldsymbol \cdot \vec n \bigr) \\ \label{psCH:conservationPhi}
			\Delta \vec \Phi & = -\langle \nabla^\bot S, \nabla \vec \Phi \rangle - \nabla^\bot \vec R \times \nabla \vec \Phi + \langle \nabla \vec \Phi, \nabla Y \rangle + \frac{\rho}{4}|\vec \Phi|^2 |\nabla \vec \Phi|^2 \vec n.
	\end{empheq}
	\emph{Step 2: Morrey decrease.} We will show that for some number $\alpha > 0$, there holds 
	\begin{equation} \label{psCH:morrey_decrease}
	\sup_{r<1/4, \, a \in B_{1/2}(0)} r^{-\alpha} \int_{B_r(a)} |\nabla \vec R|^2 + |\nabla S|^2 \, d\mathscr L^2 < \infty.
	\end{equation}
	We let $\varepsilon_0 > 0$ and fix its value later. Choose $0 < r_0 < 1/4$ such that
	\begin{equation} \label{psCH:estimate_curvature}
		\sup_{a \in B_{1/2}(0)} \int_{B_{r_0}(a)} |\nabla \vec n|^2 \, d\mathscr L^2 < \varepsilon_0.
	\end{equation}  
	Let $a$ be any point in $B_{1/2}(0)$. Denote by $\vec R_0$ the solution of 
	\begin{empheq}[left=\empheqlbrace]{align*} 
		\Delta \vec R_0 & = \Div \Bigl[\langle \vec n, \nabla Y\rangle + \frac{\rho}{4}|\vec \Phi|^2 \nabla \vec \Phi \Bigr]  && \text{in } D^2 \\
		\vec R_0 & = 0 && \text{on } \partial D^2.  
	\end{empheq} 
	Then, from \eqref{psCH:regularityXY} we obtain $\vec R_0 \in W^{2,2}(B_1(0), \mathbb R^3)$ and hence $\nabla \vec R_0 \in L^p(B_1(0), (\mathbb R^3)^2)$ for any $1 \leqslant p < \infty$. Therefore, by H\"older's inequality
	\begin{equation} \label{psCH:estimateR_0}
		\int_{B_r(a)}|\nabla \vec R_0|^2 \, d\mathscr L^2 \leqslant r \boldsymbol \alpha(2)^{1/2}\Bigl(\int_{B_1(0)} |\nabla \vec R_0|^4 \, d\mathscr L^2\Bigl)^{1/2} =: r C_1 
	\end{equation}
	whenever $0 < r \leqslant r_0$ where $\boldsymbol \alpha(2)$ is the area of the unit disk. Let $0 < r \leqslant r_0$ and let $\vec \Psi_{\vec R}$ and $\Psi_S$ be the solutions of 
 	\begin{empheq}[left=\empheqlbrace]{align*}
		\Delta \vec \Psi_{\vec R} & = \langle \nabla^\bot \vec n, \nabla S \rangle + \nabla^\bot \vec n \times \nabla \vec R && \text{in } B_{r}(a) \\
		\vec \Psi_{\vec R} & = 0 && \text{on } \partial B_{r}(a)
	\end{empheq}
	and
	\begin{empheq}[left=\empheqlbrace]{align*}
		\Delta \Psi_{S} & = \nabla^\bot \vec n \boldsymbol \cdot  \nabla \vec R && \text{in } B_{r}(a) \\
		\Psi_{S} & = 0 && \text{on } \partial B_{r}(a).
	\end{empheq}
	Then, the maps
	\begin{equation*}
		\vec \nu_{\vec R} : = \vec R - \vec R_0 - \vec \Psi_{\vec R}, \qquad \nu_{S} : = S - \Psi_S 
	\end{equation*}
	are harmonic and satisfy
	\begin{equation*}
	\vec \nu_{\vec R} = \vec R - \vec R_0, \quad \vec \nu_{S} = S \qquad \text{on } \partial B_r(a). 
	\end{equation*}
	Therefore, by monotonicity (see for instance \cite[Lemma 7.10]{MR3524220}), the Dirichlet principle, and \eqref{psCH:estimateR_0}
	\begin{equation} \label{psCH:estimate_harmonic_rest}
	\begin{split}
		\int_{B_{r/3}(a)} |\nabla \vec \nu_{\vec R}|^2 + |\nabla \nu_{S}|^2 \, d\mathscr L^2 & \leqslant \frac{1}{9} \int_{B_{r}(a)} |\nabla (\vec R - \vec R_0)|^2 + |\nabla S|^2 \, d\mathscr L^2 \\
		& \leqslant \frac{2}{9} \int_{B_{r}(a)} |\nabla \vec R |^2 + |\nabla S|^2 \, d\mathscr L^2 + \frac{2}{9}rC_1.
	\end{split}
	\end{equation}
	By Wente's theorem (see for instance \cite[Theorem 3.7]{MR3524220}) and the definition of $r_0$ \eqref{psCH:estimate_curvature} we find
	\begin{equation} \label{psCH:estimate_by_Wente}
	\begin{split}
		& \int_{B_{r}(a)}|\nabla \vec \Psi_{\vec R}|^2 + |\nabla \Psi_S|^2 \, d\mathscr L^2  \\
		& \quad \leqslant C_2\int_{B_{r_0}(a)} |\nabla \vec n|^2 \, d\mathscr L^2 \int_{B_{r}(a)} |\nabla \vec R|^2 + |\nabla S|^2 \, d\mathscr L^2 \\
		& \quad \leqslant C_2\varepsilon_0\int_{B_{r}(a)} |\nabla \vec R|^2 + |\nabla S|^2 \, d\mathscr L^2
	\end{split}
	\end{equation}
	for some constant $0 < C_2 < \infty$ independent of $r, r_0, \vec \Psi_{\vec R}, \Psi_{S}$. Using the inequalities \eqref{psCH:estimateR_0}--\eqref{psCH:estimate_by_Wente} we compute
	\begin{align*}
		& \int_{B_{r/3}(a)} |\nabla \vec R|^2 + |\nabla S|^2 \, d\mathscr L^2 \\
		& \quad \leqslant 3\int_{B_{r/3}(a)}|\nabla \vec \Psi_{\vec R}|^2 + |\nabla \Psi_S|^2 \, d\mathscr L^2 \\
		& \qquad + 3\int_{B_{r/3}(a)} |\nabla \vec \nu_{\vec R}|^2 + |\nabla \nu_{S}|^2 \, d\mathscr L^2 + 3\int_{B_{r/3}(a)}|\nabla \vec R_0|^2 \, d\mathscr L^2 \\
		& \quad \leqslant (3C_2\varepsilon_0 + 6/9) \int_{B_{r}(a)} |\nabla \vec R|^2 + |\nabla S|^2 \, d\mathscr L^2 + \frac{6}{9} r C_1 + r C_1. 
	\end{align*}
	Therefore, taking $\varepsilon_0 = (3C_29)^{-1}$ yields 
	\begin{equation} \label{psCH:induction_start}
		\int_{B_{r/3}(a)} |\nabla \vec R|^2 + |\nabla S|^2 \, d\mathscr L^2 \leqslant \frac{7}{9}\int_{B_{r}(a)} |\nabla \vec R|^2 + |\nabla S|^2 \, d\mathscr L^2 + r 2C_1
	\end{equation}
	for all $0<r\leqslant r_0$.  We next show by induction that 
	\begin{equation} \label{psCH:induction_claim}
	\begin{split}
		&\int_{B_{3^{-n}r_0}(a)} |\nabla \vec R|^2 + |\nabla S|^2 \, d\mathscr L^2 \\
		&\quad \leqslant \Bigl(\frac{7}{9}\Bigr)^n \int_{B_{r_0}(a)} |\nabla \vec R|^2 + |\nabla S|^2 \, d\mathscr L^2 + r_02C_1 \sum_{i = 1}^n 3^{-i+1}\Bigl(\frac{7}{9}\Bigr)^{n-i}
	\end{split}
	\end{equation}
	for all $n \in \mathbb N$. Indeed, letting 
	\begin{equation*}
		A(s) := \int_{B_{s}(a)} |\nabla \vec R|^2 + |\nabla S|^2 \, d\mathscr L^2 \qquad \text{for } 0 < s \leqslant r_0
	\end{equation*}
	we have from \eqref{psCH:induction_start} that $A(r_0/3^1) \leqslant (\frac{7}{9})^1 A(r_0) + r_02C_1 \sum_{i = 1}^1 3^{-i+1}\bigl(\frac{7}{9}\bigr)^{1-i}$. Assuming \eqref{psCH:induction_claim} to be true for some integer $n$, we get from \eqref{psCH:induction_start} that  
	\begin{align*}
		A(r_0/3^{n+1}) & \leqslant \frac{7}{9}A(r_0/3^n) + r_03^{-n}2C_1 \\
		& \leqslant \frac{7}{9} \Bigl[\Bigl(\frac{7}{9}\Bigr)^n A(r_0) + r_02C_1 \sum_{i = 1}^n 3^{-i+1}\Bigl(\frac{7}{9}\Bigr)^{n-i} \Bigr] + r_03^{-n}2C_1 \\
		& \leqslant \Bigl(\frac{7}{9}\Bigr)^{n+1} A(r_0) + r_02C_1 \sum_{i = 1}^{n + 1} 3^{-i+1}\Bigl(\frac{7}{9}\Bigr)^{n + 1 - i}.
	\end{align*}
	Thus, by induction, \eqref{psCH:induction_claim} holds true for all $n \in \mathbb N$.  Since
	\begin{equation*}
		2C_1 \sum_{i = 1}^n 3^{-i+1}\Bigl(\frac{7}{9}\Bigr)^{n-i} \leqslant \Bigl(\frac{7}{9}\Bigr)^{n} 2C_1 3 \sum_{i = 1}^n \frac{9^i}{3^i7^i} \leqslant \Bigl(\frac{7}{9}\Bigr)^{n} 12 C_1,
	\end{equation*}
	it follows that 
	\begin{equation*}
		\int_{B_{3^{-n}r_0}(a)} |\nabla \vec R|^2 + |\nabla S|^2 \, d\mathscr L^2 \leqslant \left(\frac{r_0}{3^n}\right)^\alpha C_0
	\end{equation*}
	for $\alpha = \log_3(9/7)$ and
	\begin{equation*}
		C_0 = r_0^{- \alpha} \Bigl(\int_{B_1(0)} |\nabla \vec R|^2 + |\nabla S|^2 \, d\mathscr L^2 + 12C_1\Bigr)
	\end{equation*}
	which implies \eqref{psCH:morrey_decrease} as $C_0$ and $\alpha$ are independent of $a$.  From \eqref{psCH:conservationR}, \eqref{psCH:conservationS}, the definition of $\vec R_0$, and H\"older's inequality it follows
	\begin{equation*}
		\sup_{r<1/4, \, a \in B_{1/2}(0)} r^{-\alpha/2} \int_{B_r(a)} |\Delta(\vec R - \vec R_0)| + |\Delta S| \, d\mathscr L^2 < \infty
	\end{equation*}
	and hence, by a classical estimate on Riesz potentials \cite{MR0458158},
	\begin{equation*}
		\nabla (\vec R - \vec R_0) \in L_{\mathrm{loc}}^p(B_{1/2}(0), (\mathbb R^3)^2), \qquad \nabla S \in L_{\mathrm{loc}}^p(B_{1/2}(0), \mathbb R^2)
	\end{equation*}
	for some $p>2$. Since $\nabla \vec R_0 \in L^q(B_{1/2}(0), (\mathbb R^3)^2)$ for all $1 \leqslant q < \infty$, we obtain
	\begin{equation} \label{psCH:regularityRS}
		\nabla \vec R \in L_{\mathrm{loc}}^p(B_{1/2}(0), (\mathbb R^3)^2), \qquad \nabla S \in L_{\mathrm{loc}}^p(B_{1/2}(0), \mathbb R^2).
	\end{equation}
	\emph{Step 3: Bootstrapping.} Putting \eqref{psCH:regularityXY} and \eqref{psCH:regularityRS} into \eqref{psCH:conservationPhi}, we infer
	\begin{equation*}
		\nabla \vec n \in L^p_{\mathrm{loc}}(B_{1/2}(0), (\mathbb R^3)^2),
	\end{equation*}
	for some $p>2$ given in the previous step. By H\"older's inequality and \eqref{psCH:conservationR}, \eqref{psCH:conservationS} we first get
	\begin{equation*}
		|\Delta (\vec R -\vec R_{0})| \in L^q_{\mathrm{loc}}(B_{1/2}(0)), \qquad |\Delta S| \in L^q_{\mathrm{loc}}(B_{1/2}(0))
	\end{equation*}
	for  $q:= p/2 > 1$ and then, by Sobolev embedding,
	\begin{equation*}
		\nabla (\vec R-\vec R_{0}) \in L_{\mathrm{loc}}^{q^*}(B_{1/2}(0), (\mathbb R^3)^2), \qquad \nabla S \in L_{\mathrm{loc}}^{q^*}(B_{1/2}(0), \mathbb R^2)
	\end{equation*}
	where $q^* := 2q/(2 - q) = 2p/(4 - p)$ satisfies  $q^* > 2q = p$ as $p > 2$. 
	\\Since $\nabla \vec R_{0} \in L_{\mathrm{loc}}^{q}(B_{1/2}(0)$ for all $1 \leqslant q < \infty$, we infer 
	\begin{equation*}
		\nabla \vec R \in L_{\mathrm{loc}}^{q^*}(B_{1/2}(0), (\mathbb R^3)^2), \qquad \nabla S \in L_{\mathrm{loc}}^{q^*}(B_{1/2}(0), \mathbb R^2).
	\end{equation*}
	 Notice that $q^*$ as above induces a recursively defined sequence of real numbers. Given a starting point $q_0 > 1$, this sequence is unbounded as $q^* > 2q$. Hence, we can repeat this procedure to obtain 
	 	\begin{equation*}
		\nabla  \vec R \in L^q_{\mathrm{loc}}(B_{1/2}(0), (\mathbb R^3)^2), \quad \nabla S \in L^q_{\mathrm{loc}}(B_{1/2}(0), \mathbb R^2) \qquad \text{for all } 1 \leqslant q <\infty.
	\end{equation*}
	Therefore, from the system of conservation laws \eqref{psCH:conservationR}--\eqref{psCH:conservationPhi} we get step by step for all $1 \leqslant q < \infty$
	\begin{gather*}
		\vec \Phi \in W^{2,q}_{\mathrm{loc}}(B_{1/2}(0), \mathbb R^3), \qquad \nabla \vec n \in L^q_{\mathrm{loc}}(B_{1/2}(0), (\mathbb R^3)^2), \qquad Y \in W^{2,q}_{\mathrm{loc}}(B_{1/2}(0), \mathbb R),  \\
		\vec R \in W^{2,q}_{\mathrm{loc}}(B_{1/2}(0), \mathbb R^3), \qquad S \in W^{2,q}_{\mathrm{loc}}(B_{1/2}(0), \mathbb R).
	\end{gather*}
	Iteration gives
	\begin{equation*}
		\vec \Phi \in W^{k,p}_{\mathrm{loc}}(B_{1/2}(0), \mathbb R^3) \qquad \text{for all } k \in \mathbb N, \, 1 \leqslant p < \infty
	\end{equation*}
	and hence,
	\begin{equation*}
		\vec \Phi \in C^\infty(B_{1/2}(0))
	\end{equation*}
	which finishes the proof. 
\end{proof}

Let $A_0, V_0 > 0$ satisfy the isoperimetric inequality: $A_0^3 \geqslant 36 \pi V_0^2$ and let 
\begin{equation*}
\mathcal{F}_{A_0, V_0} := \mathcal F \cap \{\vec \Phi: \area \vec \Phi = A_0, \, \vol \vec \Phi = V_0\}
\end{equation*}
be the family of weak (possibly branched) immersions with area $A_{0}$ and enclosed volume $V_{0}$. Using the scaling invariance of the Willmore energy (which implies the equivalence between a scale invariant isoperimetric-ratio constraint versus  a double constraint on enclosed volume and area), from \cite[Lemma 2.1]{MR2928137} we get that
$$ \inf_{\vec{\Phi} \in \mathcal {F}_{A_{0},V_{0}} } \int_{\mathbb S^2} H ^{2} \, d\mu_{\vec{\Phi}} < 8 \pi. $$
Define
\begin{equation}\label{eq:defeps0}
	\varepsilon_{\ref{lem:embeddedness}}(A_{0},V_{0}) : =   \dfrac{ \sqrt{8 \pi} - \sqrt{\inf_{\vec{\Phi} \in \mathcal {F}_{A_{0},V_{0}} } \int_{\mathbb S^2} H ^{2} \, d\mu_{\vec{\Phi}}}}{2 \sqrt{A_{0}}} \in \left(0, (\sqrt{2}-1) \sqrt{\pi}/\sqrt{A_0} \right],
\end{equation}
where the upper bound is given by the Willmore Theorem  \cite[Theorem 7.2.2]{MR1261641}, see also \eqref{pre:lower_bound_Willmore}. Notice that  $\varepsilon_{\ref{lem:embeddedness}}(A_{0},V_{0})$ depends continuously on $A_{0}$ and $V_{0}$. Indeed, from \cite[Theorem 1.1]{MR2928137}, $\inf_{\vec{\Phi} \in \mathcal {F}_{A_{0},V_{0}} } \int_{\mathbb S^2} H ^{2} \, d\mu_{\vec{\Phi}}$ is a continuous function of the isoperimetric ratio.

\begin{lemma} \label{lem:embeddedness}
	Let $A_{0} > 0, V_{0} >0$ satisfy the isoperimetric inequality: $A_{0}^{3} \geqslant 36 \pi V_{0}^{2}$ and let $\varepsilon_{\ref{lem:embeddedness}}=\varepsilon_{\ref{lem:embeddedness}}(A_{0},V_{0})$ be defined as in \eqref{eq:defeps0}.
	Then, for any $c_{0} \in (-\varepsilon_{\ref{lem:embeddedness}}, \varepsilon_{\ref{lem:embeddedness}})$, the following holds.
	\\Any minimizing sequence $\vec{\Phi}_{k}$ of $\inf_{\vec{\Phi} \in \mathcal {F}_{A_{0},V_{0}} } \int_{\mathbb S^2} ( H - c_{0} )^{2} \, d\mu_{\vec{\Phi}}$  satisfies 
	\begin{equation}
	\limsup_{k \rightarrow + \infty} \int_{\mathbb S^2} H_{\vec \Phi_k}^{2} \, d\mu_{\vec{\Phi}_{k}} < 8 \pi. \nonumber
	\end{equation}
\end{lemma} 

\begin{proof}
	From the Cauchy-Schwartz inequality, we have $\vert \int_{\mathbb S^2} H \, d\mu_{\vec{\Phi}} \vert \leqslant \sqrt{\int_{\mathbb S^2} H^{2} \, d\mu_{\vec{\Phi}}} \sqrt{\textrm{Area}({\vec{\Phi}})}$.  Thus:
	\begin{equation*}
	\left( \sqrt{\int_{\mathbb S^2} H^{2} \, d\mu_{\vec{\Phi}}} - \vert c_{0} \vert \sqrt{ \textrm{Area}({\vec{\Phi}})  } \right)^{2} \leqslant \int_{\mathbb S^2} (H-c_{0})^{2} \, d\mu_{\vec{\Phi}} \leqslant \left( \sqrt{\int_{\mathbb S^2} H^{2} \, d\mu_{\vec{\Phi}}} + \vert c_{0} \vert \sqrt{ \textrm{Area}({\vec{\Phi}})} \right)^{2}
	\end{equation*}
	which yields
	\begin{equation}\label{eq:HelWil}
	\left \vert \sqrt{\int_{\mathbb S^2} H^{2} \, d\mu_{\vec{\Phi}}}-  \sqrt{\int_{\mathbb S^2} (H-c_{0})^{2} \, d\mu_{\vec{\Phi}}} \right\vert \leqslant \vert c_{0} \vert \sqrt{ \textrm{Area}({\vec{\Phi}})}. 
	\end{equation}
	In particular, we deduce that
	\begin{equation}\label{eq:infWHH0}
	\left \vert \inf_{\vec{\Phi} \in \mathcal{F}_{A_{0},V_{0}}} \sqrt{ \int_{\mathbb S^2} H^{2} \, d\mu_{\vec{\Phi}}} -  \inf_{\vec{\Phi} \in \mathcal{F}_{A_{0},V_{0}}} \sqrt{\int_{\mathbb S^2} (H-c_{0})^{2} \, d\mu_{\vec{\Phi}}} \right \vert
	\leqslant \vert c_{0} \vert \sqrt{A_{0}}. 
	\end{equation}
	Let $c_{0}\in (- \varepsilon_{\ref{lem:embeddedness}}(A_{0},V_{0}), \varepsilon_{\ref{lem:embeddedness}}(A_{0},V_{0}))$ and let $\vec{\Phi}_{k}$ be a minimizing sequence  of $\inf_{\vec{\Phi} \in \mathcal {F}_{A_{0},V_{0}} } \int_{\mathbb S^2} ( H - c_{0} )^{2} \, d\mu_{\vec{\Phi}}$. 
	\\For $k$ large enough it holds
	\begin{equation}\label{eq:PhikInf}
	\sqrt{\int_{\mathbb S^2} ( H_k - c_{0} )^{2} \, d\mu_{\vec{\Phi}_{k}}} \leqslant \sqrt{\inf_{\vec{\Phi} \in \mathcal{F}_{A_{0},V_{0}}} \int_{\mathbb S^2} (H-c_{0})^{2} \, d\mu_{\vec{\Phi}}}+ (\varepsilon_{\ref{lem:embeddedness}}(A_{0},V_{0})- \vert c_{0} \vert) \sqrt{A_{0}} .
	\end{equation}
	Combining \eqref{eq:HelWil}, \eqref{eq:PhikInf}, and \eqref{eq:infWHH0}, we get
	\begin{align}
	\displaystyle{\sqrt{\int_{\mathbb S^2} H_k^{2} \, d\mu_{\vec{\Phi}_{k}}}} & \leqslant \sqrt{\int_{\mathbb S^2} (H_k- c_{0})^{2} \, d\mu_{\vec{\Phi}_{k}}}  +  \vert c_{0} \vert \sqrt{A_{0}} \nonumber \\
	& \leqslant \sqrt{\inf_{\vec{\Phi} \in \mathcal{F}_{A_{0},V_{0}}} \int_{\mathbb S^2} (H-c_{0})^{2} \, d\mu_{\vec{\Phi}}} + (\varepsilon_{\ref{lem:embeddedness}}(A_{0},V_{0})- \vert c_{0} \vert) \sqrt{A_{0}} +   \vert c_{0} \vert \sqrt{A_{0}}  \nonumber \\
	& <  \sqrt{\inf_{\vec{\Phi} \in \mathcal{F}_{A_{0},V_{0}}} \int_{\vec{\Phi}_{k}} H^{2} \, d\mu_{\vec{\Phi}}} + 2\varepsilon_{\ref{lem:embeddedness}}(A_{0},V_{0}) \sqrt{A_{0}} =  \sqrt{8 \pi}, \nonumber
	\end{align}
	where in the last identity we plugged in the definition of $\varepsilon_{\ref{lem:embeddedness}}(A_{0},V_{0}) $ as in \eqref{eq:defeps0}. 
\end{proof}

{\footnotesize
\bibliographystyle{alpha}
\bibliography{mybib}}

%
%\medskip \noindent \textbf{Affiliation} \medskip
%
%\noindent \textsc{Andrea Mondino} \\
%\href{mailto:A.Mondino@warwick.ac.uk}{A.Mondino@warwick.ac.uk} \newline
%
%\noindent \textsc{Christian Scharrer} \\
%\href{mailto:C.Scharrer@warwick.ac.uk}{C.Scharrer@warwick.ac.uk} \newline
%
%\noindent University of Warwick \newline 
%Mathematics Institute \newline
%Zeeman Building \newline
%Coventry CV4 7AL \newline
%GREAT BRITAIN
\end{document}